\documentclass[a4paper]{amsart}
\author[Kuzeljevi{\' c}]{Bori{\v s}a Kuzeljevi{\' c}}
\thanks{Significant part of this paper has been written while the first author visited National University of Singapore. The first author would like to thank NUS for hospitality.
The first author was also partially supported by the Science Fund of the Repubic of Serbia grant number 7750027 (SMART)}
\address[Kuzeljevi{\' c}]{Department of Mathematics and Informatics\\
Faculty of Sciences\\
University of Novi Sad}
\email{\href{borisha@dmi.uns.ac.rs}{borisha@dmi.uns.ac.rs}}
\urladdr{\url{https://people.dmi.uns.ac.rs/\textasciitilde borisha}}
\author[Raghavan]{Dilip Raghavan}
\thanks{Second author was partially supported by the Singapore Ministry of Education's research grant number A-8001467-00-00}
\address[Raghavan]{Department of Mathematics\\
National University of Singapore\\
Singapore 119076.}
\email{\href{dilip.raghavan@protonmail.com}{dilip.raghavan@protonmail.com}}
\urladdr{\url{https://dilip-raghavan.github.io/}}
\date{\today}
\title[Order structure]{Order structure of P-point ultrafilters and their relatives}
\usepackage[only, llbracket, rrbracket]{stmaryrd}
\usepackage[only, xRightarrow]{empheq}
\usepackage{amssymb, amsmath, amsthm, mathrsfs, enumitem, amsfonts, latexsym, bbm}
\usepackage[hidelinks]{hyperref}
\def\polhk#1{\setbox0=\hbox{#1}{\ooalign{\hidewidth
    \lower1.5ex\hbox{`}\hidewidth\crcr\unhbox0}}}
\newtheorem{Theorem}{Theorem}[section]

\newtheorem{Lemma}[Theorem]{Lemma}
\newtheorem{Cor}[Theorem]{Corollary}
\newtheorem{conj}[Theorem]{Conjecture}
\newtheorem{Question}[Theorem]{Question}

\theoremstyle{definition}
\newtheorem{Def}[Theorem]{Definition}

\theoremstyle{remark}
\newtheorem{remark}[Theorem]{Remark}

\newcommand{\restrict}{\mathord\upharpoonright}
\newcommand{\forallbutfin}{{\forall}^{\infty}}

\newcommand{\cc}{\mathfrak{c}}

\newcommand{\dd}{{\mathfrak{d}}}
\newcommand{\rrr}{{\mathfrak{r}}}

\newcommand{\homr}{{\mathfrak{hom}}_{\mathrm{R}}}
\newcommand{\homh}{{\mathfrak{hom}}_{\mathrm{MT}}}

\newcommand{\uu}{{\mathfrak{u}}}

\renewcommand{\[}{\left[}
\renewcommand{\]}{\right]}

\newcommand{\lc}{\left|}
\newcommand{\rc}{\right|}

\newcommand\ZFC{\mathrm{ZFC}}
\newcommand\FIN{\mathrm{FIN}}
\newcommand\MA{\mathrm{MA}}

\newcommand\CH{\mathrm{CH}}

\newcommand{\BS}{{\omega}^{\omega}}

\DeclareMathOperator{\col}{Col}

\DeclareMathOperator{\cov}{cov}

\DeclareMathOperator{\is}{IS}
\DeclareMathOperator{\dom}{dom}
\DeclareMathOperator{\ran}{ran}

\DeclareMathOperator{\cf}{cf}

\DeclareMathOperator{\rk}{rk}
\newcommand{\Pset}{\mathcal{P}}
\newcommand{\MMM}{\mathcal{M}}

\newcommand{\BBB}{\mathcal{B}}

\newcommand{\CCC}{{\mathcal{C}}}

\newcommand{\GGG}{{\mathcal{G}}}

\newcommand{\UUU}{{\mathcal{U}}}
\newcommand{\VVV}{{\mathcal{V}}}
\newcommand{\WWW}{{\mathcal{W}}}
\newcommand{\HHH}{{\mathcal{H}}}

\newcommand{\FFF}{{\mathcal{F}}}

\newcommand{\V}{{\mathbf{V}}}

\newcommand{\VG}{{{\mathbf{V}}[G]}}

\newcommand{\RR}{\mathbb{R}}

\newcommand{\ih}{{\mathcal{I}}_{\mathtt{Hindman}}}
\newcommand{\XXX}{\mathcal{X}}
\newcommand{\KKK}{\mathcal{K}}

\newcommand{\pr}[2]{\left\langle #1, #2 \right\rangle}
\newcommand{\seqq}[4]{\left\langle {#1}_{#2}: #2 #3 #4 \right\rangle}

\newcommand{\pc}[2]{{\[#1\]}^{#2}}
\newcommand{\lb}[2]{#1 \; {<}_{\mathtt{b}} \; #2}
\newcommand{\lbb}{{<}_{\mathtt{b}}}

\newcommand{\br}[1]{\left(#1\right)}
\newcommand{\brq}[1]{\left[#1\right]}
\newcommand{\set}[1]{\left\{#1\right\}}
\newcommand{\abs}[1]{\left\vert#1\right\vert}
\newcommand{\seq}[1]{\left<#1\right>}
\newcommand{\sset}{\operatorname{set}}
\newcommand{\cu}{\mathcal U}
\newcommand{\cv}{\mathcal V}
\begin{document}
\begin{abstract}
  We survey some recent results about the order structure of various kinds of ultrafilters.
  More precisely, we study Rudin-Keisler and Tukey reducibility in classes of selective, stable ordered-union, and P-point ultrafilters.
  Although these reductions are fundamentally different, there are connections between them.
  On the other hand, even though the classes of ultrafilters we consider are similar, there are significant differences in their order structure, as will be seen in the survey.
\end{abstract}
\maketitle
\section{Introduction}

The purpose of this survey is to present some recent results about the order structure of certain classes of ultrafilters on $\omega$.
For all undefined notions we refer the reader to Section \ref{s:notation}.
The term \emph{order structure} in this paper is reserved almost exclusively for Rudin-Keisler and Tukey reducibility of ultrafilters.
Note that these orders are defined in a completely different manner.
The Rudin-Keisler order is defined in terms of the existence of a function from $\omega$ to $\omega$, while the Tukey order is defined in terms of the existence of a function from the powerset of $\omega$ to the powerset of $\omega$.
Another significant difference, that will be discussed in the paper, is the fact that $\ZFC$ proves that there are two Rudin-Keisler incomparable ultrafilters, whereas it is not known whether $\ZFC$ proves that there are two Tukey incomparable ultrafilters.
All the classes of ultrafilters we will consider in the paper are actually consistent counterexamples to the latter question, known as the Isbell's problem.
Thus $\ZFC$ proves that neither of these ultrafilters exists.
Note also that \cite{natasha_survey} is an excellent survey about Tukey types of ultrafilters on a countable set.

One class of ultrafilters we will mention, in Section \ref{s:p}, is the class of P-point ultrafilters.
There are many equivalent definitions of a P-point ultrafilter, but in essence these are ultrafilters which are very close to being countably closed.
This is the key property which makes them the most natural counterexample to Isbell's problem.
Another class we will be looking into, in Section \ref{sec:selective}, is the class of selective ultrafilters.
These are P-point ultrafilters with additional properties.
One way to describe them is as minimal ultrafilters in the Rudin-Keisler ordering.
The other way to describe them is as those ultrafilters which contain witnesses to all the instances of the infinite Ramsey's theorem for pairs.
There are many other equivalent definitions of selective ultrafilters, as well.

The third class of ultrafilters we consider, in Section \ref{sec:sou}, is the class of stable ordered-union ultrafilters.
These ultrafilters arise naturally from the Milliken-Taylor's theorem, in a similar manner as selective ultrafilters come from Ramsey's theorem.
There is a formal way to explain this, which we will cover in the paper.

Finally, in the last section (Section \ref{sec:weakppoints}), we comment on weakenings of P-points.
One weakening with respect to the Tukey reducibility, and the other weakening with respect to being generic over a certain partial order.
Section \ref{s:notation} contains all the neccessary definitions and preliminary lemmas, as well as most of the needed notation.
Note however, that at some places in the paper, where it seemed more natural, notions were defined as they apear in the text.

\section{Preliminaries}\label{s:notation}

The notation is mostly standard, so $\subseteq$ denotes the subset relation, while $A\subsetneq B$ means that $A\subseteq B$ and $A\neq B$.
Similarly, $A\subseteq^* B$ means that $A\setminus B$ is a finite set.
For a set $X$, the powerset of $X$ is $\mathcal P(X)$.
Cardinality of a set $A$ is denoted by $\abs{A}$, and $\omega=\set{0,1,2,\dots}$ is the first infinite cardinal.
If $A$ is a set and $\kappa$ is a cardinal, then $[A]^{\kappa}=\set{X\subseteq A: \abs{X}=\kappa}$ and $[A]^{<\kappa}=\set{X\subseteq A: \abs{X}<\kappa}$.
If $f$ is a function and $X\subseteq \dom(f)$, then $f''X=\set{f(x):x\in X}$ denotes the direct image of the set $X$ under the map $f$. On the other hand, when $Y\subseteq \ran(f)$, then $f^{-1}(Y)=\set{x: f(x)\in Y}$ denotes the inverse image of the set $Y$ under the map $f$.
For sets $A$ and $B$, the set of all functions from $A$ to $B$ is denoted $B^A$.
For two sequences $a\in X^{\alpha}$ and $b\in Y^{\beta}$, their concatenation is denoted $a^{\frown}b$.

\begin{Def}\label{d:ultrafilter}
	For a set $X$, we say that $\cu\subsetneq \mathcal P(X)$ is \emph{an ultrafilter} on $X$ if:
	\begin{enumerate}
		\item $X\in\cu$,
		\item $a\cap b\in \cu$ for any $a,b\in \cu$,
		\item $b\in \cu$ whenever there is $a\in \cu$ such that $a\subseteq b$,
		\item for each $a\subseteq X$, either $a\in \cu$ or $X\setminus a\in\cu$.
	\end{enumerate}
\end{Def}

Recall that an ultrafilter $\cu$ on a set $X$ is \emph{principal} if there is $x\in X$ such that $\cu=\set{a\subseteq X:x\in a}$.
Otherwise, we say that $\cu$ is \emph{non-principal}.
All ultrafilters are assumed to be non-principal unless otherwise stated. The one exception to this rule will be $\beta X$. By definition, $\beta X$ contains all ultrafilters on $X$, including the principal ones. The remainder, which consists of the non-principal ultrafilters, is denoted $\beta X \setminus X$.

In the case of $\omega$, the space of all ultrafilters on $\omega$ is denoted $\beta\omega$, while the space of non-principal ultrafilters on $\omega$ is denoted $\omega^*=\beta\omega\setminus\omega$.
There is a natural topology on $\beta\omega$ which makes it a compact Hausdorff space homeomorphic with the Stone-Cech compactification of $\omega$ as a discrete space.
Namely, basic open sets are of the form $A^*=\set{\cu\in\beta\omega:A\in\cu}$ for $A\subseteq\omega$.
More details can be found in \cite{vanmill}, for example.
Now we define central objects of our study, P-points and two kinds of orders on ultrafilters.

\begin{Def}\label{d:ppoint}
	Let $\cu$ be a non-principal ultrafilter on $\omega$.  We say that $\cu$ is a \emph{P-point} ultrafilter if for every collection $\set{a_n:n<\omega}\subseteq\cu$ there is an $a\in\cu$ such that $a\subseteq^*a_n$ for every $n<\omega$.
\end{Def}

There are several equivalent definitions of being a P-point. One is that $\cu$ is a P-point if and only if for any function $f:\omega\to\omega$ there is a set $a\in\cu$ such that f is either constant or finite-to-one on $a$.
The other is immediate from the definition, an ultrafilter $\cu$ on $\omega$ is a P-point if and only if the intersection of any countably many neighborhoods of $\cu$ in $\beta\omega$ contains a neighborhood of $\cu$.

Next, we move to the Rudin-Keisler reducibility.
Recall that if $X$ is a nonempty set, then $\mathcal F\subsetneq \mathcal P(X)$ is a filter on $X$ if it satisfies conditions (1-3) of Definition \ref{d:ultrafilter}.  

\begin{Def}
	Let $X$ and $Y$ be non-empty sets, $\mathcal F$ a filter on $X$, and $\mathcal G$ a filter on $Y$. We say that $\mathcal F$ is \emph{Rudin-Keisler reducible} to $\mathcal G$ if there is a function $f:Y\to X$ such that for every $a\subseteq X$:
	$$a\in\mathcal F\Leftrightarrow f^{-1}(a)\in\mathcal G.$$
	In this case, we also say that $\mathcal F$ is Rudin-Keisler below $\mathcal G$ and write $\mathcal F\le_{RK}\mathcal G$.
\end{Def}

\begin{remark}\label{r:downwardPpoint}
  Note that if $\cu$ and $\cv$ are ultrafilters on $\omega$ and $f:\omega\to\omega$ is such that $f''a\in\cu$ for each $a\in\cv$, then $f$ already witnesses that $\cu\le_{RK}\cv$.
	Note also that if $\cu$ is a P-point and $\cv$ is an ultrafilter such that $\cv\le_{RK}\cu$, then $\cv$ is a P-point.
\end{remark}

Let $(D,\le_D)$ be a partially ordered set. We say that $D$ is \emph{directed} if for every $x$ and $y$ in $D$ there is $z$ in $D$ such that $x\le_D z$ and $y\le_D z$.
To compare the complexity of directed sets one typically uses the notion of a Tukey map.
Recall that for directed sets $D$ and $E$, a map $f:D\to E$ is \emph{a Tukey map} if for every unbounded set $X\subseteq D$ the set $f''X$ is unbounded in $E$.
Equivalently, the preimage under $f$ of every bounded set in $E$, is bounded in $D$.
When there is a Tukey map from $D$ to $E$, we say that $D$ is \emph{Tukey reducible} to $E$ and write $D\le_T E$.
It is well known that there is a Tukey map $f:D\to E$ if and only if there is a convergent map $g:E\to D$, i.e. the image under $g$ of every cofinal subset of $E$ is cofinal in $D$.

Note that any ultrafilter $\cu$ on $\omega$ can be viewed as a directed set, as $(\cu,\supseteq)$ ordered by $\supseteq$ relation.
In the case of ultrafilters, Tukey reducibility has a simpler form.
If $\cu$ and $\cv$ are ultrafilters, then $\cu\le_T \cv$ if and only if there is a map $\phi:\cv\to\cu$ which is monotone and cofinal in $\cu$.
Recall that for two families of sets $\mathcal X$ and $\mathcal Y$, a map $\phi:\mathcal X\to\mathcal Y$ is said to be \emph{monotone} if $\phi(a)\subseteq \phi(b)$ whenever $a,b\in\mathcal X$ and $a\subseteq b$, while we say that $\phi$ is \emph{cofinal in} $\mathcal Y$ if for every $b\in \mathcal Y$ there is $a\in\mathcal X$ such that $\phi(a)\subseteq b$.  

We denote $\FIN=[\omega]^{<\omega}\setminus\set{\emptyset}$.
If $(X,<)$ is a well-order and $\operatorname{otp}(X,<)=\abs{X}$, then $X(i)$ denotes the $i$th element of $X$ for $i<\abs{X}$.
It will always be clear which ordering we use.
For example if $X$ is a subset of $\omega$, then $X(i)$ is the $i$th element of the order $(X,<)$ inherited from $(\omega,<)$.
For functions $f,g\in \omega^{\omega}$ we define a relation
$$f<^{*}g\ \mbox{if and only if}\ \exists m<\omega \forall n\ge m \brq{f(n)\le g(n)}.$$
We say that a set $F\subseteq \omega^{\omega}$ is unbounded if there is no function $g\in \omega^{\omega}$ such that $\forall f\in F \brq{f<^*g}$, i.e. if $F$ is an unbounded subset of $(\omega^{\omega},<^*)$.
A collection $F\subseteq [\omega]^{\omega}$ is said to have the \emph{finite intersection property (FIP)} if $\bigcap A$ is infinite for every finite $A\subseteq F$.
Now we define some of the cardinal invariants of the continuum.
The first one is \emph{the pseudointersection number:}
$$\mathfrak p=\min\set{\abs{F}:F\subseteq [\omega]^{\omega}\wedge F\ \mbox{has the FIP}\wedge \neg \exists b\in [\omega]^{\omega} \forall a\in F \brq{b\subseteq^*a}}.$$
The second one is \emph{the bounding number:}
$$\mathfrak b=\min\set{\abs{F}:F\ \mbox{is an unbounded subset of}\ (\omega^{\omega},<^*)}.$$
The third one is \emph{the dominating number:}
$$\mathfrak d=\min\set{\abs{F}:F\subseteq \omega^{\omega}\ \wedge\ \forall g\in \omega^{\omega} \exists f\in F\ [g<^*f]},$$
i.e. $\mathfrak d$ is the minimal size of a dominating family of functions in $\omega^{\omega}$.  

To formulate results in the rest of the paper, we will often need certain set theoretic assumptions. Before we state them, recall a couple of notions.
The cardinality of the continuum is $\mathfrak c=2^{\aleph_0}$.
A poset $\mathbb{P}$ is \emph{ccc} if there is no uncountable collection of pairwise incompatible conditions in $\mathbb{P}$.
A poset $\mathbb{P}$ is \emph{$\sigma$-centered} if it can be written as the countable union $\mathbb{P}=\bigcup_{n<\omega}P_n$ of centered subsets, where $X\subseteq \mathbb{P}$ is \emph{centered} if any finitely many members of $X$ have a lower bound in $\mathbb{P}$. The assumptions we will be using are:

\medskip
$\CH$:\ \ $2^{\aleph_0}=\aleph_1$.

\medskip
$\MA$:\ \ For every ccc partial order $\mathbb{P}$, every $\alpha<\mathfrak{c}$, and every collection $\set{D_i:i<\alpha}$ of sets dense in $\mathbb{P}$, there is a filter $G\subseteq \mathbb{P}$ such that $G\cap D_i\neq\emptyset$ for every $i<\alpha$.

\medskip
$\MA(\sigma-\mbox{centered})$:\ \ For every $\sigma$-centered poset $\mathbb{P}$, every $\alpha<\mathfrak{c}$, and every collection $\set{D_i:i<\alpha}$ of sets dense in $\mathbb{P}$, there is a filter $G\subseteq \mathbb{P}$ such that $G\cap D_i\neq\emptyset$ for every $i<\alpha$.

\begin{remark}
	Note that $\MA(\sigma-\mbox{centered})$ is equivalent to $\mathfrak p=\mathfrak c$ and that $$\CH\Rightarrow \MA\Rightarrow \MA(\sigma-\mbox{centered}).$$
\end{remark}

Note also that $\MA_{\alpha}$ is the statement that for every ccc partial order $\mathbb{P}$, and every collection $\set{D_i:i<\alpha}$ of sets dense in $\mathbb{P}$, there is a filter $G\subseteq \mathbb{P}$ such that $G\cap D_i\neq\emptyset$ for every $i<\alpha$. Then $\MA$ simply says that $\MA_{\alpha}$ holds for each $\alpha<\mathfrak{c}$.

\section{Selective ultrafilters} \label{sec:selective}
Ramsey's theorem for pairs states that if $c: \pc{\omega}{2} \rightarrow 2$ is any coloring, then there exists $H \in \pc{\omega}{{\aleph}_{0}}$ such that $c$ is constant on $\pc{H}{2}$.
One of the easiest ways to prove this statement is using an arbitrary ultrafilter $\UUU$ on $\omega$.
We will recall this well-known argument as a motivation for the definition of a selective ultrafilter.
For each $m \in \omega$, there exists ${i}_{m} \in 2$ such that ${K}_{m} = \{n > m: c(\{m, n\}) = {i}_{m} \} \in \UUU$.
And there exists $i \in 2$ such that $K = \{m \in \omega: {i}_{m} = i\} \in \UUU$.
If $\{{n}_{0}, \dotsc, {n}_{l}\} \subseteq K$, then we may choose ${n}_{l+1} \in K \cap \left( {\bigcap}_{j \leq l}{{K}_{{n}_{j}}} \right)$.
Proceeding in this way, we construct $H = \{{n}_{j}: j \in \omega\} \subseteq K$ with the property that $c$ is constantly $i$ on $\pc{H}{2}$.
The reader will notice that even though the sets $K$ and ${K}_{m}$ belong to $\UUU$, there is no guarantee that $H$ will belong to $\UUU$.
This consideration leads to the following definition.
\begin{Def} \label{def:selective}
 An ultrafilter $\UUU$ on $\omega$ is said to be \emph{selective} if for every $c: \pc{\omega}{2} \rightarrow 2$, there exists $H \in \UUU$ such that $c$ is constant on $\pc{H}{2}$.
\end{Def}
Thus selective ultrafilters contain a witness to each instance of Ramsey's theorem for pairs.
This turns out to be a very robust concept with a plethora of equivalent characterizations.
We need to introduce a few definitions in order to state these equivalences.
\begin{Def} \label{def:canonicalonomega}
 A function $f: \omega \rightarrow \omega$ is \emph{canonical on} a subset \emph{$A \subseteq \omega$} if $f$ is either constant or one-to-one on $A$.
\end{Def}
A simple consequence of Ramsey's theorem for pairs is that every $f: \omega \rightarrow \omega$ is canonical on some infinite $A \subseteq \omega$.
In the realm of ultrafilters, this consequence turns out to be strong enough to recover the general form of Ramsey's theorem for all finite dimensions and for any finite number of colors.
The theorem that every $f: \omega \rightarrow \omega$ is canonical on some infinite $A \subseteq \omega$ can be naturally broken into two parts: every $g: \omega \rightarrow \omega$ is either constant or finite-to-one on some infinite $B \subseteq \omega$; and every finite-to-one $h: \omega \rightarrow \omega$ is one-to-one on some infinite $C \subseteq \omega$.
The first part leads to the definition of a P-point.
The second part leads to the following.
\begin{Def} \label{def:qpoint}
 An ultrafilter $\UUU$ on $\omega$ is called a \emph{Q-point} if for every finite-to-one function $f: \omega \rightarrow \omega$, there exists $A \in \UUU$ such that $f$ is one-to-one on $A$.
\end{Def}
Every Q-point is rapid (see Definition \ref{d:rapid}), but there may be rapid ultrafilters which are not Q-points.
The following theorem is the result of combining the work of several people including Choquet, Kunen, and Silver.
A proof can be found in Bartoszy{\' n}ski and Judah~\cite{BJ} or Todorcevic~\cite{topics}.
\begin{Theorem} \label{thm:kunenramsey}
 The following are equivalent for any ultrafilter $\UUU$ on $\omega$:
 \begin{enumerate}
  \item \label{kunenramsey:first}
  $\UUU$ is selective;
  \item \label{kunenramsey:second}
  for each $1 \leq n, k < \omega$ and $c: \pc{\omega}{n} \rightarrow k$, there is an $A \in \UUU$ such that $c$ is constant on $\pc{A}{n}$;
  \item \label{kunenramsey:third}
  for every function $f: \omega \rightarrow \omega$, there exists $A \in \UUU$ such that $f$ is canonical on $A$;
  \item \label{kunenramsey:fourth}
  whenever $\XXX \subseteq \pc{\omega}{{\aleph}_{0}}$ is analytic, there exists $A \in \UUU$ such that either $\pc{A}{{\aleph}_{0}} \subseteq \XXX$ or $\pc{A}{{\aleph}_{0}} \cap \XXX = \emptyset$;
  \item \label{kunenramsey:fifth}
  $\UUU$ is both a P-point and a Q-point.
 \end{enumerate}
\end{Theorem}
The partition property in item (4) can be further strengthened in the presence of large cardinals to include all subsets of $\pc{\omega}{{\aleph}_{0}}$ in $\mathbf{L}(\RR)$.
This leads to the notion of a generic ultrafilter over a model.
$(\pc{\omega}{\omega}, {\subseteq}^{\ast})$ is a countably closed forcing, and hence it does not add any reals.
If $\UUU \subseteq \pc{\omega}{\omega}$ is a generic filter for this forcing over some model $\V$, then $\UUU$ is a selective ultrafilter in $\V\[\UUU\]$.
A remarkable theorem of Todorcevic (see \cite{MR1654181} and \cite{MR1644345}) says that in the presence of large cardinals, every selective ultrafilter is generic over the inner model $\mathbf{L}(\RR)$.
\begin{Theorem}[Todorcevic]
 Assume that there is a supercompact cardinal.
 $\UUU$ is a selective ultrafilter on $\omega$ if and only if $\UUU$ is a generic filter for the forcing $(\pc{\omega}{\omega}, {\subseteq}^{\ast})$ over $\mathbf{L}(\RR)$.
\end{Theorem}
In an earlier work, Blass~\cite{Bl} had obtained the same conclusion for the model $\mathbf{HOD}(\RR)$ inside a variant of the L{\' e}vy--Solovay model.
\begin{Theorem}[Corollary 11.2 of Blass~\cite{Bl}]
 Let $\kappa$ be a Mahlo cardinal in $\V$ and let $H$ be a generic filter for $\col(\omega, < \kappa)$ over $\V$.
 Then, in $\V\[H\]$, $\UUU$ is a selective ultrafilter on $\omega$ if and only if $\UUU$ is a generic filter for the the forcing $(\pc{\omega}{\omega}, {\subseteq}^{\ast})$ over ${\mathbf{HOD}(\RR)}^{\V\[H\]}$.
\end{Theorem}
Selective ultrafilters also have a useful characterization in terms of games.
This characterization was independently discovered by Galvin and McKenzie, although neither one of them seems to have published the result.
A proof can be found in Chapter VI \textsection 5 of Shelah~\cite{PIF}.
\begin{Def} \label{def:Gsel}
 Let $\UUU$ be an ultrafilter on $\omega$.
 The \emph{selectivity game on $\UUU$}, denoted \emph{${\Game}^{\mathtt{Sel}}\left(\UUU\right)$}, is a two player perfect information game in which Players I and II alternatively choose ${A}_{i}$ and ${n}_{i}$ respectively, where ${A}_{i} \in \UUU$ and ${n}_{i} \in {A}_{i}$.
 Together they construct the sequence
 \begin{align*}
  {A}_{0}, {n}_{0}, {A}_{1}, {n}_{1}, \dotsc,
 \end{align*}
 where each ${A}_{i} \in \UUU$ has been played by Player I and ${n}_{i} \in {A}_{i}$ has been chosen by Player II in response.
 Player II wins if and only if $\{{n}_{i}: i < \omega\} \in \UUU$.
\end{Def}
\begin{Theorem}[Galvin; McKenzie] \label{thm:sel}
 An ultrafilter $\UUU$ on $\omega$ is selective if and only if Player I does not have a winning strategy in ${\Game}^{\mathtt{Sel}}\left(\UUU\right)$.
\end{Theorem}
A standard diagonalization argument like the one given by Rudin in \cite{walterppoints} shows that $\CH$ implies the existence of ${2}^{\cc}$ selective ultrafilters.
Note that this is the maximum possible number of ultrafilters on a countable set.
In fact, a much weaker hypothesis than $\CH$ is sufficient to obtain an even stronger conclusion.
\begin{Def} \label{def:generic}
 Let $\KKK$ be a class of ultrafilters on a countable set $X$.
 We say that ultrafilters from $\KKK$ \emph{exist generically} if every filter base on $X$ of size less than ${2}^{{\aleph}_{0}}$ can be extended to an ultrafilter belonging to $\KKK$.
\end{Def}
The earliest result regarding the generic existence of some special class of ultrafilters seems to be Ketonen's theorem that P-points exist generically if and only if $\dd = \cc$.
Canjar~\cite{MR0993747} introduced the general concept of generic existence for classes of ultrafilters and proved the following.
\begin{Theorem}[Canjar~\cite{MR0993747}] \label{thm:selective-generic}
 Selective ultrafilters exist generically if and only if $\cov(\MMM) = \cc$.
\end{Theorem}
Therefore, it is strictly more difficult to arrange for the generic existence of selective ultrafilters than for the generic existence of P-points.
Brendle~\cite{MR1729447} provided characterizations in terms of cardinal invariants for the generic existence of many special classes of ultrafilters.
$\cov(\MMM) = \cc$ is a fairly mild hypothesis which holds in many models of set theory, for example it is a consequence of $\MA$.
Nevertheless, selective ultrafilters may fail to exist.
\begin{Theorem}[Kunen~\cite{rknonlinear, kunenweak}] \label{thm:kunen-sel}
 There are no selective ultrafilters when ${\aleph}_{2}$ random reals are added to any model of $\ZFC + \CH$.
\end{Theorem}
We will see in Section \ref{s:p} that even P-points may fail to exist.
Miller~\cite{MR0548093} showed that there are no Q-points in the Laver model.
In the Miller model~\cite{millerforcing}, there are P-points, but no Q-points, while Q-points exist in Shelah's model with no P-points (see \cite{PIF}).
Thus the two conditions in item (5) of Theorem \ref{thm:kunenramsey} are quite independent of each other.
Actually, P-points and Q-points have opposing existence conditions.
For example, there is a Q-point if $\dd = {\aleph}_{1}$, while $\dd = \cc$ implies the existence of P-points.
This immediately leads to the following conclusion: if $\cc \leq {\aleph}_{2}$, then there is either a P-point or a Q-point.
It is a long-standing open question whether this is always the case.
\begin{Question} \label{q:porq}
 Is it consistent that there are no P-points and no Q-points?
\end{Question}
An immediate consequence of item (3) of Theorem \ref{thm:kunenramsey} is that if $\VVV$ is selective and $\UUU \; {\leq}_{RK} \; \VVV$, then $\UUU \; {\equiv}_{RK} \; \VVV$.
Thus selective ultrafilters are RK-minimal among all ultrafilters, and it is not hard to see that every RK-minimal ultrafilter must be selective.
It turns out that selective ultrafilters are Tukey minimal as well.
This was proved by Raghavan and Todorcevic~\cite{tukey}.
They were able to characterize all ultrafilters on $\omega$ that are Tukey below a selective.
In fact, \cite{tukey} contains a characterization of all ultrafilters on $\omega$ that are Tukey below a given basically generated ultrafilter.
The basically generated ultrafilters constitute a much larger class than the selectives (see Definition \ref{def:basicgen} for the notion of a basically generated ultrafilter).
We will end this section by showing that the result for selective ultrafilters found in \cite{tukey} follows from the more general characterization for basically generated ultrafilters, also found in \cite{tukey}, via standard arguments.
\begin{Def} \label{def:fubini}
 For $A \subseteq \omega \times \omega$ and $m \in \omega$, $A\[m\] = \{n \in \omega: \pr{m}{n} \in A\}$.
 Let $\UUU$ and $\seqq{\UUU}{m}{\in}{\omega}$ be ultrafilters on $\omega$.
 Define
 \begin{align*}
  {\bigotimes}_{\UUU}{{\UUU}_{m}} = \left\{ A \subseteq \omega \times \omega: \left\{m \in \omega: A\[m\] \in {\UUU}_{m} \right\} \in \UUU \right\}.
 \end{align*}
It is easily seen that ${\bigotimes}_{\UUU}{{\UUU}_{m}}$ is an ultrafilter on $\omega \times \omega$.
\end{Def}
\begin{Def} \label{def:cu}
 For $\UUU \in \beta\omega \setminus \omega$, define ${\CCC}_{\alpha, \UUU}$ by induction on $\alpha < {\omega}_{1}$ as follows.
 Let ${\CCC}_{0, \UUU} = \left\{ \VVV \in \beta\omega: \VVV \; {\equiv}_{RK} \; \UUU \right\}$.
 For $\alpha > 0$, let
 \begin{align*}
 {\CCC}_{\alpha, \UUU} = \left\{ \VVV \in \beta\omega: \exists \seqq{\VVV}{n}{\in}{\omega} \in {\left( {\bigcup}_{\xi < \alpha}{{\CCC}_{\xi, \UUU}} \right)}^{\omega} \[\VVV \; {\equiv}_{RK} \; {\bigotimes}_{\UUU}{{\VVV}_{n}}\] \right\}.
\end{align*}
\end{Def}
\begin{Lemma} \label{lem:Cbelow}
 Let $\UUU$ be a selective ultrafilter on $\omega$.
 If $\VVV \in {\CCC}_{\alpha, \UUU}$ and $\WWW \; {\leq}_{RK} \; \VVV$, where $\WWW$ is an ultrafilter on $\omega$, then $\WWW \in {\CCC}_{\xi, \UUU}$, for some $\xi \leq \alpha$.
\end{Lemma}
\begin{proof}
 Induct on $\alpha$.
 When $\alpha = 0$, this follows from the RK-minimality of selective ultrafilters.
 Suppose $\alpha > 0$, $\VVV \in {\CCC}_{\alpha, \UUU}$, $\seqq{\VVV}{n}{\in}{\omega} \in {\left( {\bigcup}_{\xi < \alpha}{{\CCC}_{\xi, \UUU}} \right)}^{\omega}$, and $\WWW \; {\leq}_{RK} \; \VVV \; {\equiv}_{RK} \; {\bigotimes}_{\UUU}{{\VVV}_{n}}$.
 Let $g: \omega \times \omega \rightarrow \omega$ be a map witnessing $\WWW \; {\leq}_{RK} \; {\bigotimes}_{\UUU}{{\VVV}_{n}}$.
 For $m \in \omega$, let ${g}_{m}: \omega \rightarrow \omega$ be the map given by ${g}_{m}(n) = g(\pr{m}{n})$, for all $n \in \omega$.
 Write $A = \left\{ m \in \omega: \exists {C}_{m} \in {\VVV}_{m}\[{g}_{m} \ \text{is constant on} \ {C}_{m}\] \right\}$ and
 \begin{align*}
  B = \left\{ m \in \omega: \exists {\zeta}_{m} < \alpha \exists {\WWW}_{m} \in {\CCC}_{{\zeta}_{m}, \UUU} \forall D \in {\VVV}_{m} \[{g}_{m}''D \in {\WWW}_{m}\] \right\}.
 \end{align*}
The induction hypothesis implies that $\omega = A \cup B$.
If $A \in \UUU$, then for each $m \in A$, let ${l}_{m} \in \omega$ be so that ${g}_{m}''{C}_{m} = \{{l}_{m}\}$.
Now the map $m \mapsto {l}_{m}$, which is defined on $A$, witnesses $\WWW \; {\leq}_{RK} \; \UUU$, whence $\WWW \in {\CCC}_{0, \UUU}$.
Hence we may assume that $B \in \UUU$.
If there exists $m \in B$ so that $\forall D \in {\VVV}_{m}\[{g}_{m}''D \in \WWW\]$, then $\WWW = {\WWW}_{m} \in {\CCC}_{{\zeta}_{m}, \UUU}$, where ${\zeta}_{m} < \alpha$.
So we may assume that for every $m \in B$, there exists ${D}_{m} \in {\VVV}_{m}$ so that ${g}_{m}''{D}_{m} \notin \WWW$.
For convenience, put ${\WWW}_{m} = \UUU \in {\CCC}_{0, \UUU}$, for all $m \notin B$.
Define $c: {\[B\]}^{2} \rightarrow 2$ by
\begin{align*}
 c(\{m, n\}) = \begin{cases}
                0 \ &\text{if} \ \exists {E}_{m, n} \in \pc{{D}_{n}}{{\aleph}_{0}} \cap {\VVV}_{n}\[{g}_{m}''{D}_{m} \cap {g}_{n}''{E}_{m, n} = \emptyset\];\\
                1 \ &\text{otherwise.}
               \end{cases}
\end{align*}
By selectivity, there exists $F \in \UUU \cap \pc{B}{{\aleph}_{0}}$ so that $c$ is constant on ${\[F\]}^{2}$.
This constant value cannot be $1$.
For otherwise, letting $m = \min(F)$ and noting that ${g}^{-1}\left( \omega \setminus {g}_{m}''{D}_{m} \right) \in {\bigotimes}_{\UUU}{{\VVV}_{k}}$, find $n \in F \setminus \{m\}$ and $E \in \pc{{D}_{n}}{{\aleph}_{0}} \cap {\VVV}_{n}$ such that $g''\left(\{n\} \times E\right) \subseteq \omega \setminus {g}_{m}''{D}_{m}$, giving a contradiction to $c(\{m, n\}) = 1$.
Therefore $c$ is constantly $0$ on ${\[F\]}^{2}$.
Define, for $n \in F$, ${G}_{n} = {D}_{n} \cap \left( {\bigcap}_{m \in F \cap n}{{E}_{m, n}} \right) \in {\VVV}_{n}$.
Note that ${g}_{n}''{G}_{n} \in {\WWW}_{n}$ and that ${g}_{n}''{G}_{n} \cap {g}_{m}''{G}_{m} = \emptyset$, for all $m \in F \cap n$.
Therefore, $H = {\bigcup}_{n \in F}{\left( \{n\} \times {g}_{n}''{G}_{n} \right)} \in {\bigotimes}_{\UUU}{{\WWW}_{n}}$ and the map $q: H \rightarrow \omega$ given by $q(\pr{n}{l}) = l$ is one-to-one on $H$.
Further, if $I \in \UUU \cap \pc{F}{{\aleph}_{0}}$ and ${J}_{n} \in {\WWW}_{n} \cap \pc{{g}_{n}''{G}_{n}}{{\aleph}_{0}}$, for $n \in I$, then $K = {\bigcup}_{n \in I}{\left(\{n\} \times {g}^{-1}_{n}({J}_{n}) \right)} \in {\bigotimes}_{\UUU}{{\VVV}_{n}}$, and $q''\left( {\bigcup}_{n \in I}{\left( \{n\} \times {J}_{n} \right)} \right) = {\bigcup}_{n \in I}{{J}_{n}} = g''K \in \WWW$.
Therefore, $q$ witnesses that $\WWW \; {\equiv}_{RK} \; {\bigotimes}_{\UUU}{{\WWW}_{n}}$, whence $\WWW \in {\CCC}_{\xi, \UUU}$, where $\xi = \sup\{{\zeta}_{m}+1: m \in B\}$.
Since $\xi \leq \alpha$, the induction is complete.
\end{proof}
\begin{Def} \label{def:rank}
 Suppose $\UUU$ is an ultrafilter and $P \subseteq \FIN$.
 $P$ is said to \emph{$\UUU$-large} if for each $A \in \UUU$, $\pc{A}{< {\aleph}_{0}} \cap P \neq \emptyset$.
 For $s \in \FIN$, ${s}^{-} = s \setminus \{\min(s)\}$.
 For $n \in \omega$, ${P}_{n} = \left\{{s}^{-}: s \in P \wedge \min(s) = n \right\}$.
 For $s, t \in \pc{\omega}{< {\aleph}_{0}}$, $s \sqsubseteq t$ means that \emph{$s$ is a non-empty initial segment of $t$} -- in other words, $s \neq \emptyset$ and $s \subseteq t$ and $\forall m \in s \forall n \in t \setminus s\[m < n\]$.
 For $t \in \pc{\omega}{< {\aleph}_{0}}$, define $\is(t) = \left\{ s \in \FIN: s \sqsubseteq t \right\}$.
 Observe that $\is(\emptyset) = \emptyset$.
 Let $Q(P) = \{t \in \pc{\omega}{< {\aleph}_{0}}: \is(t) \cap P = \emptyset\}$.
 Observe that $\emptyset \in Q(P)$.
 For $s \in Q(P)$, we say ${\rk}_{P}(s) \leq 0$ if $\{n \in \omega: s \subseteq n \wedge s \cup \{n\} \notin Q(P)\} \in \UUU$.
 When $\alpha > 0$, we say ${\rk}_{P}(s) \leq \alpha$ if ${\rk}_{P}(s) \leq 0$ or $\left\{ n \in \omega: s \subseteq n \wedge s \cup \{n\} \in Q(P) \wedge \exists {\xi}_{n} < \alpha \[{\rk}_{P}(s \cup \{n\}) \leq {\xi}_{n}\] \right\} \in \UUU$.
 Observe that if $\alpha \leq \beta$ and ${\rk}_{P}(s) \leq \alpha$, then ${\rk}_{P}(s) \leq \beta$.
 If there exists $\alpha$ such that ${\rk}_{P}(s) \leq \alpha$, then ${\rk}_{P}(s) = \min\{ \alpha: {\rk}_{P}(s) \leq \alpha \}$, and ${\rk}_{P}(s) = \infty$ otherwise.
\end{Def}
\begin{Lemma} \label{lem:urankexists}
 Suppose $\UUU$ is selective and $P \subseteq \FIN$ is $\UUU$-large.
 Then ${\rk}_{P}(\emptyset) < \infty$.
\end{Lemma}
\begin{proof}
 Suppose ${\rk}_{P}(\emptyset) = \infty$.
 Note that for any $s \in Q(P)$, if ${\rk}_{P}(s) = \infty$, then ${B}_{s} = \left\{ n \in \omega: s \subseteq n \wedge s \cup \{n\} \in Q(P) \wedge {\rk}_{P}(s \cup \{n\}) = \infty \right\} \in \UUU$.
 Therefore, it is possible to define a strategy $\Sigma$ for Player I in ${\Game}^{\mathtt{Sel}}\left(\UUU\right)$ which has the property that whenever $\left\langle \pr{{A}_{i}}{{n}_{i}}: i < \omega \right\rangle$ is a run of ${\Game}^{\mathtt{Sel}}\left(\UUU\right)$ where Player I has followed $\Sigma$, then $s \in Q(P)$ and ${\rk}_{P}(s) = \infty$, for every $s \in \pc{\left\{ {n}_{i}: i < \omega \right\}}{< {\aleph}_{0}}$.
 Since $\Sigma$ is not a winning strategy, there is such a run of ${\Game}^{\mathtt{Sel}}\left(\UUU\right)$ for which $\left\{ {n}_{i}: i < \omega \right\} \in \UUU$.
 However, since $P$ is $\UUU$-large, there exists $s \in \pc{\left\{ {n}_{i}: i < \omega \right\}}{< {\aleph}_{0}}$ such that $s \in P$, which contradicts $s \in Q(P)$.
\end{proof}
\begin{Lemma} \label{lem:pnlarge}
 Suppose $\UUU$ is selective, $P \subseteq \FIN$ is $\UUU$-large, and ${\rk}_{P}(\emptyset) \neq 0$.
 Then $\{n \in \omega: {P}_{n} \subseteq \FIN \wedge {P}_{n} \ \text{is} \ \UUU\text{-large}\} \in \UUU$.
\end{Lemma}
\begin{proof}
 Since ${\rk}_{P}(\emptyset) \not\leq 0$, $B = \{n \in \omega: \{n\} \in Q(P)\} \in \UUU$.
 Observe that ${P}_{n} \subseteq \FIN$, for every $n \in B$.
 Assume for a contradiction that $C = \{n \in B: {P}_{n} \ \text{is not} \ \UUU\text{-large}\} \in \UUU$, and for each $n \in C$, choose ${A}_{n} \in \UUU$ such that $\pc{{A}_{n}}{< {\aleph}_{0}} \cap {P}_{n} = \emptyset$.
 Since $\UUU$ is selective, there exists $D \in \UUU$ so that $D \subseteq C$ and $\forall n, m \in D\[m < n \implies n \in {A}_{m}\]$.
 Choose $s \in \pc{D}{< {\aleph}_{0}} \cap P$ and write $n = \min(s)$.
 Then $n \in C$ and ${s}^{-} \in \pc{{A}_{n}}{< {\aleph}_{0}}$, whence ${s}^{-} \notin {P}_{n}$.
 However, this contradicts $s \in P$.
\end{proof}
\begin{Lemma} \label{lem:rkbound}
 Suppose $\UUU$ is selective, $P \subseteq \FIN$ is $\UUU$-large, $n \in \omega$, ${P}_{n} \subseteq \FIN$, and ${P}_{n}$ is $\UUU$-large.
 For any $s \in Q({P}_{n})$ with $s \cap \left( n + 1 \right) = \emptyset$, ${\rk}_{{P}_{n}}(s) \leq {\rk}_{P}(\{n\} \cup s)$.
\end{Lemma}
\begin{proof}
 First observe that $\{n\} \notin P$ because ${P}_{n} \subseteq \FIN$, and so $\{n\} \cup s \in Q(P)$ because $s \in Q({P}_{n})$ and $s \cap \left( n + 1 \right) = \emptyset$.
 There is nothing to prove when ${\rk}_{P}(\{n\} \cup s) = \infty$.
 So we assume that ${\rk}_{P}(\{n\} \cup s) < \infty$ and we induct on it.
 Suppose ${\rk}_{P}(\{n\} \cup s) = 0$.
 Find $A \in \UUU$ so that for every $l \in A$, $\{n\} \cup s \subseteq l$ and $\{n\} \cup s \cup \{l\} \notin Q(P)$.
 Then for every $l \in A$, $s \subseteq l$ and $s \cup \{l\} \notin Q({P}_{n})$, whence ${\rk}_{{P}_{n}}(s) \leq 0$.
 Now suppose ${\rk}_{P}(\{n\} \cup s) = \alpha > 0$.
 Since ${\rk}_{P}(\{n\} \cup s) \not\leq 0$, let $A \in \UUU$ be so that for every $l \in A$, $\{n\} \cup s \subseteq l$, $\{n\} \cup s \cup \{l\} \in Q(P)$, and there exists ${\xi}_{l} < \alpha$ with ${\rk}_{P}(\{n\} \cup s \cup \{l\}) \leq {\xi}_{l}$.
 For every $l \in A$, $s \subseteq l$, $\left( s \cup \{l\} \right) \cap \left( n+1 \right) = \emptyset$, and $s \cup \{l\} \in Q({P}_{n})$.
 Thus the induction hypothesis applies and implies that ${\rk}_{{P}_{n}}(s \cup \{l\}) \leq {\rk}_{P}(\{n\} \cup s \cup \{l\}) \leq {\xi}_{l}$.
 Therefore, $A$ witnesses that ${\rk}_{{P}_{n}}(s) \leq \alpha$.
\end{proof}
\begin{Def} \label{def:up}
 Let $\UUU$ be an ultrafilter on $\omega$ and let $P \subseteq \FIN$.
 Define $\UUU(P) = \left\{ R \subseteq P: \exists A \in \UUU\[ P \cap \pc{A}{< {\aleph}_{0}} \subseteq R\] \right\}$.
\end{Def}
\begin{Lemma} \label{lem:upprod}
 Suppose $\UUU$ is selective and $P \subseteq \FIN$ is $\UUU$-large.
 Assume that $\forall t, s \in P\[t \subseteq s \implies t = s\]$.
 Then $\UUU(P)$ is an ultrafilter on $P$ and there exist $\xi < {\omega}_{1}$ and $\VVV \in {\CCC}_{\xi, \UUU}$ such that $\UUU(P) \; {\equiv}_{RK} \; \VVV$.
\end{Lemma}
\begin{proof}
 By Lemma \ref{lem:urankexists}, ${\rk}_{P}(\emptyset) < \infty$.
 Let $\alpha = {\rk}_{P}(\emptyset)$.
 It is clear from the definition of ${\rk}_{P}$ that $\alpha < {\omega}_{1}$.
 We show by induction on $\alpha$ that there exist $\xi \leq \alpha$ and $\VVV \in {\CCC}_{\xi, \UUU}$ such that $\UUU(P) \; {\equiv}_{RK} \; \VVV$.
 
 Assume ${\rk}_{P}(s) = 0$.
 Find $A \in \UUU$ so that for each $n \in A$, $\{n\} \notin Q(P)$, which means that $\{n\} \in P$.
 By shrinking $A$ if necessary, find a bijection $f: \omega \rightarrow P$ such that for each $n \in A$, $f(n) = \{n\}$.
 Suppose $R \subseteq P$ with ${f}^{-1}(R) \in \UUU$.
 Let $B = A \cap {f}^{-1}(R) \in \UUU$.
 Suppose $s \in \pc{B}{< {\aleph}_{0}} \cap P$ and let $n = \min(s) \in B$.
 So $\{n\} = f(n) \in R \subseteq P$.
 Since $s \in P$ and $\{n\} \subseteq s$, $s = \{n\} \in R$.
 Therefore, $\pc{B}{< {\aleph}_{0}} \cap P \subseteq R$, whence $R \in \UUU(P)$.
 Conversely, assume that $R \in \UUU(P)$, and find $B \in \UUU$ with $\pc{B}{< {\aleph}_{0}} \cap P \subseteq R \subseteq P$.
 Let $C = A \cap B \in \UUU$.
 For each $n \in C$, $f(n) = \{n\} \in \pc{B}{< {\aleph}_{0}} \cap P \subseteq R$, whence $C \subseteq {f}^{-1}(R)$, whence ${f}^{-1}(R) \in \UUU$.
 Therefore, $\UUU(P) = \left\{ R \subseteq P: {f}^{-1}(R) \in \UUU \right\}$, which shows that $\UUU(P)$ is an ultrafilter on $P$ and that $\UUU(P) \; {\equiv}_{RK} \; \UUU$ as $f$ is a bijection.
 As $\UUU \in {\CCC}_{0, \UUU}$, this is as needed.
 
 Proceeding by induction, assume that ${\rk}_{P}(\emptyset) = \alpha > 0$.
 Let $A \in \UUU$ be so that for each $n \in A$, $\{n\} \in Q(P)$ and there exists ${\xi}_{n} < \alpha$ such that ${\rk}_{P}(\{n\}) \leq {\xi}_{n}$.
 Since ${\rk}_{P}(\emptyset) \neq 0$, Lemma \ref{lem:pnlarge} applies and implies there exist $B \in \UUU$ such that $B \subseteq A$ and for each $n \in B$, ${P}_{n} \subseteq \FIN$ and ${P}_{n}$ is $\UUU$-large.
 Since $\emptyset \in Q({P}_{n})$ and $\emptyset \cap \left( n+1 \right) = \emptyset$, Lemma \ref{lem:rkbound} applies and implies that ${\rk}_{{P}_{n}}(\emptyset) \leq {\rk}_{P}(\{n\}) \leq {\xi}_{n}$, for every $n \in B$.
 Furthermore, suppose $n \in B$, and $t, s \in {P}_{n}$ are such that $t \subseteq s$.
 Find $u, v \in P$ with $\min(u) = n = \min(v)$, $t = {u}^{-}$, and $s = {v}^{-}$.
 Then $u \subseteq v$, whence $u = v$, whence $t = {u}^{-} = {v}^{-} = s$.
 Applying the inductive hypothesis, $\UUU({P}_{n})$ is an ultrafilter on ${P}_{n}$ and there are ${\zeta}_{n} \leq {\xi}_{n}$, ${\VVV}_{n} \in {\CCC}_{{\zeta}_{n}, \UUU}$, and a bijection ${f}_{n}: \omega \rightarrow {P}_{n}$ witnessing that ${\VVV}_{n} \; {\equiv}_{RK} \; \UUU({P}_{n})$, for every $n \in B$.
 By shrinking $B$ if necessary, find a bijection $f: \omega \times \omega \rightarrow P$ such that for every $\pr{n}{m} \in B \times \omega$, $f(\pr{n}{m}) = \{n\} \cup {f}_{n}(m)$.
 For convenience, define ${\VVV}_{n} = \UUU \in {\CCC}_{0, \UUU}$, for every $n \in \omega \setminus B$.
 Suppose $R \subseteq P$ with ${f}^{-1}(R) \in {\bigotimes}_{\UUU}{{\VVV}_{n}}$.
 Find $C \in \UUU \cap \pc{B}{{\aleph}_{0}}$ and ${D}_{n} \in {\VVV}_{n}$, for every $n \in C$, such that ${\bigcup}_{n \in C}\left( \{n\} \times {D}_{n} \right) \subseteq {f}^{-1}(R)$.
 For every $n \in C$, find ${E}_{n} \in \UUU$ with $\pc{{E}_{n}}{< {\aleph}_{0}} \cap {P}_{n} \subseteq {f}_{n}''{D}_{n}$.
 Since $\UUU$ is selective, there exists $E \in \UUU \cap \pc{C}{{\aleph}_{0}}$ such that $\forall m, n \in E\[m < n \implies n \in {E}_{m}\]$.
 Suppose $s \in \pc{E}{< {\aleph}_{0}} \cap P$.
 Let $n = \min(s)$.
 Then ${s}^{-} \in \pc{{E}_{n}}{< {\aleph}_{0}} \cap {P}_{n}$, whence ${s}^{-} = {f}_{n}(m)$, for some $m \in {D}_{n}$.
 Then $f(\pr{n}{m}) \in R$.
 Therefore, $s = \{n\} \cup {s}^{-} = \{n\} \cup {f}_{n}(m) = f(\pr{n}{m}) \in R$.
 This shows that $\pc{E}{< {\aleph}_{0}} \cap P \subseteq R$, which means that $R \in \UUU(P)$.
 Conversely, suppose $R \in \UUU(P)$ and fix $F \in \UUU$ with $\pc{F}{< {\aleph}_{0}} \cap P \subseteq R \subseteq P$.
 Let $G = B \cap F \in \UUU$.
 For each $n \in G$, $\pc{G}{< {\aleph}_{0}} \cap {P}_{n} \in \UUU({P}_{n})$, whence ${H}_{n} = {f}^{-1}_{n}\left( \pc{G}{< {\aleph}_{0}} \cap {P}_{n} \right) \in {\VVV}_{n}$.
 Thus ${\bigcup}_{n \in G}{\left( \{n\} \times {H}_{n} \right)} \in {\bigotimes}_{\UUU}{{\VVV}_{n}}$, and if $n \in G$ and $m \in {H}_{n}$, then $f(\pr{n}{m}) = \{n\} \cup {f}_{n}(m)$.
 Since ${f}_{n}(m) \subseteq G \subseteq F$ and $n \in F$, $f(\pr{n}{m}) \in \pc{F}{< {\aleph}_{0}} \cap P \subseteq R$.
 This shows that ${\bigcup}_{n \in G}{\left( \{n\} \times {H}_{n} \right)} \subseteq {f}^{-1}(R)$, whence ${f}^{-1}(R) \in {\bigotimes}_{\UUU}{{\VVV}_{n}}$.
 This shows $\UUU(P) = \left\{ R \subseteq P: {f}^{-1}(R) \in {\bigotimes}_{\UUU}{{\VVV}_{n}} \right\}$, which means that $\UUU(P)$ is an ultrafilter on $P$ and that ${\bigotimes}_{\UUU}{{\VVV}_{n}} \; {\equiv}_{RK} \; \UUU(P)$ as $f$ is a bijection.
 Now let $\xi = \sup\left\{ {\zeta}_{n} + 1: n \in B \right\}$, pick an arbitrary bijection $\pi: \omega \times \omega \rightarrow \omega$, and define $\WWW = \left\{ I \subseteq \omega: {\pi}^{-1}(I) \in {\bigotimes}_{\UUU}{{\VVV}_{n}} \right\}$.
 Then $\WWW$ is an ultrafilter on $\omega$ such that $\WWW \; {\equiv}_{RK} \; {\bigotimes}_{\UUU}{{\VVV}_{n}} \; {\equiv}_{RK} \; \UUU(P)$.
 Since $\WWW \in {\CCC}_{\xi, \UUU}$ and $\xi \leq \alpha$, this is as needed.
\end{proof}
\begin{Theorem}[Theorem 17 of \cite{tukey}] \label{thm:basic}
 Let $\UUU$ be basically generated.
 Let $\VVV$ be an arbitrary ultrafilter with $\VVV \; {\leq}_{T} \; \UUU$.
 Then there is $P \subseteq \FIN$ such that:
 \begin{enumerate}
  \item
  $\forall t, s \in P\[t \subseteq s \implies t = s\]$;
  \item
  $\UUU(P) \; {\equiv}_{T} \; \UUU$;
  \item
  $\VVV \; {\leq}_{RK} \; \UUU(P)$.
 \end{enumerate}
\end{Theorem}
\begin{Cor}[Lemma 22 and Theorem 24 of \cite{tukey}] \label{cor:selectivetukey}
 Let $\UUU$ be a selective ultrafilter.
 Suppose $\VVV$ is an ultrafilter on $\omega$ such that $\VVV \; {\leq}_{T} \; \UUU$.
 Then $\VVV \in {\CCC}_{\xi, \UUU}$, for some $\xi < {\omega}_{1}$.
\end{Cor}
\begin{proof}
 Selective ultrafilters are basically generated, in fact, they are basic (see \cite{dt}).
 So apply Theorem \ref{thm:basic} to find $P \subseteq \FIN$ satisfying (1)--(3) of that theorem.
 Note that if $P$ is not $\UUU$-large, then $\UUU(P) = \Pset(P)$.
 Since $\pr{\UUU}{\supseteq} \; {\not\equiv}_{T} \; \pr{\Pset(P)}{\supseteq}$, $P$ must be $\UUU$-large for (2) of Theorem \ref{thm:basic} to hold.
 Therefore Lemma \ref{lem:upprod} implies that there exist $\alpha < {\omega}_{1}$ and $\WWW \in {\CCC}_{\alpha, \UUU}$ such that $\UUU(P) \; {\equiv}_{RK} \; \WWW$.
 Thus by (3) of Theorem \ref{thm:basic}, $\VVV \; {\leq}_{RK} \; \UUU(P) \; {\equiv}_{RK} \; \WWW$.
 Now by Lemma \ref{lem:Cbelow}, $\VVV \in {\CCC}_{\xi, \UUU}$, for some $\xi \leq \alpha < {\omega}_{1}$.
\end{proof}
The final sentence of Section 5 of \cite{tukey} pointed out that Theorem 24 of \cite{tukey} could be derived as a consequence of Theorem 17 of \cite{tukey}.
However, details of this argument were not provided there.
\section{Stable ordered-union ultrafilters} \label{sec:sou}
Graham and Rothschild~\cite{MR284352} conjectured that whenever $\omega$ is partitioned into finitely many pieces, then one of the pieces contains all distinct sums from some infinite subset.
Hindman~\cite{MR349574} proved this conjecture and this result is known as Hindman's theorem.
In the earlier paper~\cite{MR307926}, Hindman had established a connection with certain ultrafilters on $\omega$: the Graham and Rothschild Conjecture holds if and only if there is an ultrafilter on $\omega$ such that every member of it contains all non-repeating sums from some infinite subset.
In the same paper he also proved that under the Continuum Hypothesis ($\CH$), the conjecture of Graham and Rothschild is equivalent to the existence of an ultrafilter $\UUU$ which is an idempotent -- i.e.\@ satisfying $\UUU + \UUU = \UUU$ -- in the semigroup $(\beta\omega,+)$, where ultrafilters are thought of as finitely additive measures and the $+$ operation is given by the convolution of measures.
Not withstanding these equivalences, Hindman managed to find an elementary, but technical, proof of the Graham and Rothschild Conjecture in~\cite{MR349574}, avoiding the use of ultrafilters completely.
Nevertheless, it was quickly observed by Galvin and Glazer (see Baumgartner~\cite{MR354394}), that a short proof of Hindman's theorem could be given by using a lemma of Ellis~\cite{MR101283} about the existence of idempotents in arbitrary semigroups.
Indeed, this is the standard proof of Hindman's theorem given in texts on Ramsey theory today (e.g.\@ Todorcevic~\cite{introramsey}).

The various connections with ultrafilters discovered by Hindman in~\cite{MR307926} were soon elaborated and explored by others.
In~\cite{MR349574} and~\cite{MR2991425}, van Douwen is credited with the observation that if $\CH$ holds, then there is an ultrafilter on $\omega$ which has a base consisting of sets made up of all the non-repeating sums from some infinite subset of $\omega$.
Such ultrafilters are called \emph{strongly summable ultrafilters}, and van Douwen raised the question of whether such ultrafilters exist in $\ZFC$.
Blass~\cite{blass-hindman} introduced ordered-union ultrafilters as a counterpart to van Douwen's strongly summable ones for another result of Hindman from \cite{MR349574} about the union operation on finite subsets of $\omega$.
We introduce some notation in order to state this theorem.
\begin{Def} \label{def:block}
 $\FIN$ denotes the collection of non-empty finite subsets of $\omega$.
 For $s, t \in \FIN$, write $\lb{s}{t}$ to mean $\max(s) < \min(t)$.
 $X \subseteq \FIN$ is called a \emph{block sequence} if $X$ is non-empty and it is linearly ordered by the relation $\lbb$.
 The notation $X(i)$ will be used to denote the $i$th member of $\pr{X}{\lbb}$, for all $i < \lc X \rc$.
 For $1 \leq \alpha \leq \omega$ and $A \subseteq \FIN$,
 \begin{align*}
  &{A}^{\[\alpha\]} = \left\{ X \subseteq A: X \ \text{is a block sequence and} \ \lc X \rc = \alpha \right\};\\
  &{A}^{\[< \alpha\]} = \left\{ X \subseteq A: X \ \text{is a block sequence and} \ 1 \leq \lc X \rc < \alpha \right\}.
 \end{align*}
 Thus ${A}^{\[\alpha\]}$ is the collection of all block sequences of length $\alpha$ from $A$ and ${A}^{\[< \alpha\]}$ is the collection of all block sequences of length $< \alpha$ from $A$.
 For $A \subseteq \FIN$,
 \begin{align*}
  \[A\] = \left\{\bigcup{X}: X \in {A}^{\[< \omega\]}\right\}.
 \end{align*}
 Thus $\[A\]$ is the collection of all unions of finite length block sequences from $A$.
 Note $A \subseteq \[A\] \subseteq \FIN$.
\end{Def}
In \cite{MR349574}, Hindman proved that for every $1 \leq n < \omega$ and $c: \FIN \rightarrow n$, there exists $X \in {\FIN}^{\[\omega\]}$ such that $c$ is constant on $\[X\]$.
This result can be proved using an idempotent in the semigroup $(\gamma\FIN, \cup)$, and it is easily seen to imply Hindman's theorem about non-repeating sums via the map $s \mapsto \sum_{n\in s}2^n$, for $s \in \FIN$.
For this reason, Hindman's result on finite unions is also referred to as Hindman's theorem.
\begin{Def} \label{def:Ih}
 Define $\ih = \left\{ A \subseteq \FIN: \neg \exists X \in {\FIN}^{\[\omega\]}\[\[X\] \subseteq A \]\right\}$.
 Hindman's theorem implies that $\ih$ is a proper non-principal ideal on $\FIN$.
\end{Def}
Suppose $\HHH$ is an ultrafilter on $\FIN$ so that $\HHH \cap \ih = \emptyset$.
Then, by definition, for every $A \in \HHH$, there exists $X \in {\FIN}^{\[\omega\]}$ with $\[X\] \subseteq A$, and indeed this condition is equivalent to the condition that $\HHH \cap \ih = \emptyset$.
\begin{Def} \label{def:gammafinu}
 Let $\gamma\FIN = \left\{\HHH \in \beta\FIN: \forall k \in \omega \[ \left\{s \in \FIN: k < \min(s) \right\} \in \HHH \] \right\}$.
 It is clear that $\gamma\FIN$ is a closed subset of $\beta\FIN$ and that every $\HHH \in \gamma\FIN$ is non-principal.
 For $\GGG$ and $\HHH$ in $\gamma\FIN$ define
 \begin{align*}
  \GGG \cup \HHH = \left\{ A \subseteq \FIN: \left\{ s \in \FIN: \left\{ t \in \FIN: \lb{s}{t} \ \text{and} \ s \cup t \in A \right\} \in \HHH \right\} \in \GGG \right\}.
 \end{align*}
 It is not difficult to show that $\GGG \cup \HHH \in \gamma\FIN$ and that $(\gamma\FIN, \cup)$ is a compact semigroup (see Lemmas 2.13 to 2.16 of \cite{introramsey}).
 $\HHH$ is said to be an \emph{idempotent in $(\gamma\FIN, \cup)$} if $\HHH \cup \HHH = \HHH$.
\end{Def}
Suppose $\HHH$ is an idempotent in $(\gamma\FIN, \cup)$ and consider any $A \in \HHH$.
Say that $s \in A$ is \emph{good} if $\{\lb{s}{t}: s \cup t \in A\} \in \HHH$.
As $A \in \HHH \cup \HHH$, $\left\{s \in A: s \ \text{is good} \right\} \in \HHH$.
If $s \in A$ is good, then by idempotence, there exists $B \in \HHH$ so that for each $t \in B$, $\lb{s}{t}$, $s \cup t \in A$, and $\{\lb{t}{u}: s \cup t \cup u \in A\} \in \HHH$, which implies that $s \cup t$ is good.
So for every good $s \in A$, $\{ \lb{s}{t}: s \cup t \ \text{is good} \} \in \HHH$.
Now construct $X \in {\FIN}^{\[\omega\]}$ by induction as follows.
Let $X(0) \in A$ be good.
Assume that $\{X(0), \dotsc, X(n)\}$ are given so that every $s \in \[\{X(0), \dotsc, X(n)\}\]$ is good, which means that ${B}_{s} = \{\lb{X(n)}{t}: s \cup t \ \text{is good}\} \in \HHH$.
Choose $X(n+1) \in {\bigcap}_{s \in \[\{X(0), \dotsc, X(n)\}\]}{{B}_{s}}$.
Then $\lb{X(n)}{X(n+1)}$ and every $s \in \[\{X(0), \dotsc, X(n), X(n+1)\}\]$ is good, allowing the induction to proceed.
This construction shows that every idempotent $\HHH$ in $(\gamma\FIN, \cup)$ has the property that $\forall A \in \HHH \exists X \in {\FIN}^{\[\omega\]}\[\[X\] \subseteq A\]$.
Ordered-union ultrafilters satisfy a strengthening of the idempotence condition.
\begin{Def}[Blass~\cite{blass-hindman}] \label{def:ou}
 An ultrafilter $\HHH$ on $\FIN$ is called \emph{ordered-union} if for every $c: \FIN \rightarrow 2$ there is $X \in {\FIN}^{\[\omega\]}$ so that $\[X\] \in \HHH$ and $c$ is constant on $\[X\]$.
\end{Def}
Note that an ultrafilter on $\FIN$ is ordered-union if and only if it has a filter base consisting of sets of the form $\[X\]$ where $X \in {\FIN}^{\[\omega\]}$.
So these ultrafilters contain a witness to each instance of Hindman's theorem.
Ramsey's theorem and Hindman's theorem have a common generalization, the Milliken--Taylor theorem of~\cite{MR373906} and~\cite{MR424571}.
The Milliken--Taylor theorem says that for any $1 \leq n, k < \omega$ and $c: {\FIN}^{\[n\]} \rightarrow k$, there exists $X \in {\FIN}^{\[\omega\]}$ such that $c$ is constant on ${\[X\]}^{\[n\]}$.
\begin{Def}[Blass~\cite{blass-hindman}] \label{def:sou}
 An ultrafilter $\HHH$ on $\FIN$ is called \emph{stable ordered-union} if for every $c: {\FIN}^{\[2\]} \rightarrow 2$, there exists $X \in {\FIN}^{\[\omega\]}$ such that $\[X\] \in \HHH$ and $c$ is constant on ${\[X\]}^{\[2\]}$.
\end{Def}
So these ultrafilters contain a witness to every instance of the Milliken-Taylor theorem.
Suppose $\HHH$ is ordered-union and consider some $A \in \HHH$.
Let $X \in {\FIN}^{\[\omega\]}$ be so that $\[X\] \in \HHH$ and $\[X\] \subseteq A$.
Then for every $s \in \[X\]$, ${B}_{s} = \{t \in \[X\]: \lb{s}{t}\} \in \HHH$ and for each $t \in {B}_{s}$, $s \cup t \in A$, showing that $A \in \HHH \cup \HHH$.
Thus every ordered-union ultrafilter is an idempotent in $(\gamma\FIN, \cup)$.
Hence, we have the following implications for an ultrafilter $\HHH$ on $\FIN$.
\begin{align*}
 &\HHH \ \text{is stable ordered-union} \ \xRightarrow{\mathrm{(I)}} \HHH \ \text{is ordered-union} \ \xRightarrow{\mathrm{(II)}}\\
 &\HHH \ \text{is an idempotent in} \ (\gamma\FIN, \cup) \xRightarrow{\mathrm{(III)}} \forall A \in \HHH \exists X \in {\FIN}^{\[\omega\]}\[\[X\] \subseteq A\].
\end{align*}
To see that $\xRightarrow{\mathrm{(III)}}$ is not reversible, let $\{{X}_{n}: n \in \omega\}$ be members of ${\FIN}^{\[\omega\]}$ such that $\[{X}_{n}\] \cap \[{X}_{m}\] = \emptyset$, for every $m < n < \omega$.
For each $n$, let ${\HHH}_{n}$ be an ultrafilter on $\FIN$ with $\[{X}_{n}\] \in {\HHH}_{n}$ and ${\HHH}_{n} \cap \ih = \emptyset$.
Let $\UUU$ be any non-principal ultrafilter on $\omega$ and let $\HHH$ be the ultrafilter $\left\{ A \subseteq \FIN: \left\{ n \in \omega: A \cap \[{X}_{n}\] \in {\HHH}_{n} \right\} \in \UUU \right\}$.
It is easy to see that $\HHH$ is not an idempotent, but has the property that $\forall A \in \HHH \exists X \in {\FIN}^{\[\omega\]}\[\[X\] \subseteq A\]$.
Unlike idempotents, the min and max projections of an ordered union ultrafilter are tightly constrained.
\begin{Def} \label{def:minmax}
 Let $\HHH$ be an ultrafilter on $\FIN$.
 Define
 \begin{align*}
  {\HHH}_{\min} &= \left\{ M \subseteq \omega: \left\{s \in \FIN: \min(s) \in M \right\} \in \HHH \right\}\\
  {\HHH}_{\max} &= \left\{ M \subseteq \omega: \left\{s \in \FIN: \max(s) \in M \right\} \in \HHH \right\}.
 \end{align*}
\end{Def}
Clearly ${\HHH}_{\min}$ and ${\HHH}_{\max}$ are ultrafilters on $\omega$, and the maps $\min: \FIN \rightarrow \omega$ and $\max: \FIN \rightarrow \omega$ witness that ${\HHH}_{\min}, {\HHH}_{\max} \; {\leq}_{RK} \; \HHH$.
${\HHH}_{\min}$ and ${\HHH}_{\max}$ are usually called the \emph{min} and the \emph{max projections of $\HHH$}.
\begin{Theorem}[Blass~\cite{blass-hindman} and Blass and Hindman~\cite{MR906807}] \label{thm:selectiveprojections}
 Let $\HHH$ be an ordered-union ultrafilter on $\FIN$.
 Then ${\HHH}_{\min}$ and ${\HHH}_{\max}$ are selective ultrafilters on $\omega$ such that ${\HHH}_{\min} \; {\not\equiv}_{RK} \; {\HHH}_{\max}$.
\end{Theorem}
Theorem \ref{thm:selectiveprojections} shows that $\xRightarrow{\mathrm{(II)}}$ does not reverse because it is an easy consequence of Ellis' Lemma that for every non-principal $\UUU \in \beta\omega$ there is an idempotent $\HHH$ in $(\gamma\FIN, \cup)$ with $\UUU = {\HHH}_{\min}$.
It also shows that unlike idempotents, whose existence is a theorem of $\ZFC$, ordered-union ultrafilters may fail to exist.
The question of whether $\xRightarrow{\mathrm{(I)}}$ can be reversed is a long-standing basic open problem in this area.
\begin{Question} \label{q:1}
 Is every ordered-union ultrafilter stable?
\end{Question}
Question \ref{q:1} is attributed to Blass from the 1980s.
If the answer to this question is ``no'', then one can further ask whether there is a model of $\ZFC$ with ordered-union ultrafilters, but no stable ones.
Blass~\cite{blass-hindman} showed that the stable ordered-union ultrafilters are the analogues of the selective ultrafilters for ${\FIN}^{\[\omega\]}$.
In particular, they can be characterized in terms of canonical forms for functions from $\FIN$ to $\omega$, and they are generically added by a countably closed forcing notion on ${\FIN}^{\[\omega\]}$.
\begin{Def} \label{def:refines}
 Let $X, Y \in {\FIN}^{\[\omega\]}$.
 $Y$ is said to \emph{refine} $X$ if $\forall i \in \omega\[Y(i) \in \[X\]\]$.
 We write $Y \leq X$ to denote this relation.
 $Y$ is said to \emph{almost refine} $X$ if $\forallbutfin i \in \omega\[Y(i) \in \[X\]\]$.
 This relation is denoted by $Y \; {\leq}^{\ast} \; X$.
\end{Def}
\begin{Def} \label{def:canonical}
 A function $f: \FIN \rightarrow \omega$ is \emph{canonical on} a subset $A \subseteq \FIN$ if one of the following statements hold:
 \begin{enumerate}
  \item
  $\forall s, t \in A\[f(s)=f(t)\]$;
  \item
  $\forall s, t \in A\[f(s)=f(t) \leftrightarrow \min(s)=\min(t)\]$;
  \item
  $\forall s, t \in A\[f(s)=f(t) \leftrightarrow \max(s)=\max(t)\]$;
  \item
  $\forall s, t \in A\[ f(s)=f(t) \leftrightarrow \left( \min(s)=\min(t) \wedge \max(s)=\max(t) \right) \]$;
  \item
  $\forall s, t \in A\[f(s)=f(t) \leftrightarrow s=t\]$.
 \end{enumerate}
\end{Def}
\begin{Theorem}[Theorem 4.2 of Blass~\cite{blass-hindman}] \label{thm:blassramsey}
 The following are equivalent for any ultrafilter $\HHH$ on $\FIN$:
 \begin{enumerate}
  \item \label{blassramsey:first}
  $\HHH$ is stable ordered-union;
  \item \label{blassramsey:second}
  for each $1 \leq n, k < \omega$ and $c: {\FIN}^{\[n\]} \rightarrow k$, there is an $X \in {\FIN}^{\[\omega\]}$ such that $\[X\] \in \HHH$ and $c$ is constant on ${\[X\]}^{\[n\]}$; 
  \item \label{blassramsey:third}
  for every function $f: \FIN \rightarrow \omega$, there exists $X \in {\FIN}^{\[\omega\]}$ such that $\[X\] \in \HHH$ and $f$ is canonical on $\[X\]$;
  \item \label{blassramsey:fourth}
  whenever $\XXX \subseteq {\FIN}^{\[\omega\]}$ is analytic, there exists $X \in {\FIN}^{\[\omega\]}$ such that $\[X\] \in \HHH$ and either ${\[X\]}^{\[\omega\]} \subseteq \XXX$ or ${\[X\]}^{\[\omega\]} \cap \XXX = \emptyset$;
  \item \label{blassramsey:fifth}
  $\HHH$ is ordered-union and for every sequence $\seqq{X}{n}{\in}{\omega}$ with the property that for all $n \in \omega$, ${X}_{n} \in {\FIN}^{\[\omega\]}$ and $\[{X}_{n}\] \in \HHH$, there exists $Y \in {\FIN}^{\[\omega\]}$ such that $\forall n \in \omega\[Y \; {\leq}^{\ast} \; {X}_{n}\]$ and $\[Y\] \in \HHH$.
 \end{enumerate}
\end{Theorem}
Notice that Theorem \ref{thm:blassramsey} is almost identical to Theorem \ref{thm:kunenramsey}.
The main difference is in item (\ref{blassramsey:fifth}).
Whereas item (\ref{kunenramsey:fifth}) of Theorem \ref{thm:kunenramsey} requires the two independent conditions of being a P-point and a Q-point, item (\ref{blassramsey:fifth}) of Theorem \ref{thm:blassramsey} states a condition that is deceptively similar to the definition of P-point, although Theorem \ref{thm:blassramsey} is showing this condition to be the analogue of selectivity for ordered-union ultrafilters.
Of course, in the unlikely event that the answer to Question \ref{q:1} is ``yes'', this condition would be redundant.
The partition property given by item (\ref{blassramsey:fourth}) can be strengthened further in the presence of large cardinals to cover all subsets of ${\FIN}^{\[\omega\]}$ in $\mathbf{L}(\RR)$.
This provides the impetus to consider generic ultrafilters added by the forcing $({\FIN}^{\[\omega\]}, {\leq}^{\ast})$.
To elaborate, $({\FIN}^{\[\omega\]}, {\leq}^{\ast})$ is a countably closed forcing notion, and hence, it does not add any new reals.
If $G \subseteq {\FIN}^{\[\omega\]}$ is a generic filter over some transitive universe $\V$, then it is easy to see that $\HHH = \{A \subseteq \FIN: \exists X \in G\[\[X\] \subseteq A\]\}$ is a stable ordered-union ultrafilter in $\VG$.
We will say that $\HHH$ is the \emph{ultrafilter added by $G$} if it has this form.
Using the same ideas found in \cite{MR1644345} and \cite{topics}, Todorcevic proved that if there are sufficiently large cardinals, then every stable ordered-union ultrafilter is added by some generic $G$ over $\mathbf{L}(\RR)$.
\begin{Theorem}[Todorcevic]
 Assume that there is a supercompact cardinal.
 $\HHH$ is a stable ordered-union ultrafilter on $\FIN$ if and only if $\HHH$ is added by some generic filter for the forcing $({\FIN}^{\[\omega\]}, {\leq}^{\ast})$ over $\mathbf{L}(\RR)$.
\end{Theorem}
In earlier work, Blass had used the arguments from \cite{Bl} to derive the same conclusion for $\mathbf{HOD}(\RR)$ inside a variant of the L{\' e}vy--Solovay model.
\begin{Theorem}[Blass]
 Let $\kappa$ be a Mahlo cardinal in $\V$ and let $H$ be a generic filter for $\col(\omega, < \kappa)$ over $\V$.
 Then, in $\V\[H\]$, $\HHH$ is a stable ordered-union ultrafilter on $\FIN$ if and only if $\HHH$ is added by some generic filter for the the forcing $({\FIN}^{\[\omega\]}, {\leq}^{\ast})$ over ${\mathbf{HOD}(\RR)}^{\V\[H\]}$.
\end{Theorem}
Just like their selective counterparts, the stable ordered-union ultrafilters have a useful characterization in terms of two-player games.
\begin{Def}
 Let $\HHH$ be any ultrafilter on $\FIN$.
 The \emph{stability game on $\HHH$}, denoted \emph{${\Game}^{\mathtt{Stab}}(\HHH)$}, is a two player game in which Players I and II alternatively choose sets ${A}_{i}$ and ${s}_{i}$ respectively, where ${A}_{i} \in \HHH$ and ${s}_{i} \in {A}_{i}$.
 During a run of the game, they construct the sequence
 \begin{align*}
  {A}_{0}, {s}_{0}, {A}_{1}, {s}_{1}, \dotsc,
 \end{align*}
 where each ${A}_{i} \in \HHH$ has been played by Player I and ${s}_{i} \in {A}_{i}$ has been chosen by Player II in response.
 Player II wins this run if and only if $\forall i < j < \omega\[\lb{{s}_{i}}{{s}_{j}}\]$ and $\[\left\{ {s}_{i}: i < \omega \right\}\] \in \HHH$.
\end{Def}
\begin{Theorem}[see Lemma 2.13 of \cite{souvssu}] \label{thm:gstab}
 An ultrafilter $\HHH$ on $\FIN$ is stable ordered-union if and only if Player I does not have a winning strategy in ${\Game}^{\mathtt{Stab}}(\HHH)$.
\end{Theorem}
It turns out that stable ordered-union ultrafilters exist generically under the same circumstances as selective ultrafilters.
\begin{Theorem}[Eisworth~\cite{MR1889561}] \label{thm:sou-generic}
 Stable ordered-union ultrafilters exist generically if and only if $\cov(\MMM) = \cc$.
\end{Theorem}
Combining Theorems \ref{thm:selective-generic} and \ref{thm:sou-generic} we conclude that the generic existence of selective ultrafilters is equivalent to the generic existence of stable ordered-union ultrafilters.
This further accentuates the question of whether stable ordered-union ultrafilters are any harder to construct than selective ultrafilters.
Recall that in Theorem \ref{thm:selectiveprojections} Blass had shown that the existence of one stable ordered-union ultrafilter guarantees the existence of at least two RK-non-isomorphic selective ultrafilters.
In \cite{blass-hindman}, Blass proved that under $\CH$ any two RK-non-isomorphic selective ultrafilters are realized as the min and max projections of some stable ordered-union ultrafilter.
\begin{Theorem}(Theorem 2.4 of Blass~\cite{blass-hindman}) \label{thm:blass-ch}
 Assume $\CH$, and let $\UUU$ and $\VVV$ be selective ultrafilters such that $\UUU \; {\not\equiv}_{RK} \; \VVV$.
 Then there is a stable ordered-union ultrafilter $\HHH$ such that ${\HHH}_{\max} = \UUU$ and ${\HHH}_{\min} = \VVV$.
\end{Theorem}
Blass~\cite{blass-hindman} raised the question of whether the existence of a stable ordered-union ultrafilter follows from the existence of at least two RK-non-isomorphic selective ultrafilters.
This long-standing question of Blass was recently answered negatively by Raghavan and Stepr{\= a}ns~\cite{souvssu}.
\begin{Theorem}[Raghavan and Stepr{\= a}ns~\cite{souvssu}] \label{thm:souvssu}
 There is a model of $\ZFC$ with ${2}^{{\aleph}_{0}} = {\aleph}_{2}$ pairwise RK-non-isomorphic selective ultrafilters on $\omega$ and no stable ordered-union ultrafilters on $\FIN$.
\end{Theorem}
This theorem shows that it is provably harder to produce one ultrafilter that contains a witness to each instance of the Milliken--Taylor theorem than it is to produce many ultrafilters containing witnesses to every instance of Ramsey's theorem.
We do not know whether there are any ordered-union ultrafilters in the model from \cite{souvssu}, nor do we know if there are any selective ultrafilters of character ${\aleph}_{2}$.
\begin{Question} \label{q:souvssu}
 Is it consistent to have ${2}^{{\aleph}_{0}}$ pairwise RK-non-isomorphic selective ultrafilters and no ordered-union ultrafilters?
 Is it consistent to have ${2}^{{\aleph}_{0}}$ pairwise RK-non-isomorphic selective ultrafilters of character ${2}^{{\aleph}_{0}}$ and no stable ordered-union ultrafilters?
\end{Question}
Blass characterized all ultrafilters that are RK below a stable ordered-union ultrafilter.
Using (\ref{blassramsey:third}) of Theorem \ref{thm:blassramsey}, he was able to show that there are precisely 4 such ultrafilters.
One of these is the following.
\begin{Def} \label{def:Hminmax}
 For an ultrafilter $\HHH$ on $\FIN$, define
 \begin{align*}
  {\HHH}_{\mathrm{minmax}} = \left\{A \subseteq \omega \times \omega: \left\{s \in \FIN: \pr{\min(s)}{\max(s)} \in A \right\} \in \HHH \right\}.
 \end{align*}
\end{Def}
${\HHH}_{\mathrm{minmax}}$ is easily seen to be an ultrafilter on $\omega \times \omega$ and it is clear that the map $\pr{\min}{\max}: \FIN \rightarrow \omega \times \omega$ witnesses ${\HHH}_{\mathrm{minmax}} \; {\leq}_{RK} \; \HHH$.
Indeed, Blass used Theorem \ref{thm:selectiveprojections} to show that ${\HHH}_{\mathrm{minmax}}$ is RK-isomorphic to the product of the min projection of $\HHH$ with its max projection.
\begin{Lemma}[Blass~\cite{blass-hindman}] \label{lem:minmaxprod}
 If $\HHH$ is a stable ordered-union ultrafilter on $\FIN$, then ${\HHH}_{\mathrm{minmax}} \; {\equiv}_{RK} \; {\HHH}_{\min} \bigotimes {\HHH}_{\max}$.
\end{Lemma}
\begin{Theorem}[Blass~\cite{blass-hindman}] \label{thm:sourkbelow}
 Suppose that $\HHH$ is a stable ordered-union ultrafilter on $\FIN$.
 If $\UUU$ is an ultrafilter on $\omega$ such that $\UUU \; {\leq}_{RK} \; \HHH$, then $\UUU \; {\equiv}_{RK} \; \HHH$, or $\UUU \; {\equiv}_{RK} \; {\HHH}_{\min} \bigotimes {\HHH}_{\max}$, or $\UUU \; {\equiv}_{RK} \; {\HHH}_{\min}$, or $\UUU \; {\equiv}_{RK} \; {\HHH}_{\max}$.
\end{Theorem}
\begin{proof}
 Let $f: \FIN \rightarrow \omega$ witness $\UUU \; {\leq}_{RK} \; \HHH$.
 Use (\ref{blassramsey:third}) of Theorem \ref{thm:blassramsey} to find $A \in \HHH$ such that $f$ is canonical on $A$.
 Since $\UUU$ is non-principal, alternative (1) of Definition \ref{def:canonical} cannot hold.
 If alternative (2) holds, then there exists $g: \omega \rightarrow \omega$ such that $g$ is one-to-one on $\{\min(s): s \in A\}$ and $\forall s \in A\[f(s) = g(\min(s))\]$.
 Now, $g$ witnesses ${\HHH}_{\min} \; {\equiv}_{RK} \; \UUU$.
 Similarly, if (3) holds, then ${\HHH}_{\max} \; {\equiv}_{RK} \; \UUU$.
 If (4) holds, then $\UUU \; {\equiv}_{RK} \; {\HHH}_{\mathrm{minmax}} \; {\equiv}_{RK} \; {\HHH}_{\min} \bigotimes {\HHH}_{\max}$ by Lemma \ref{lem:minmaxprod}.
 Finally if (5) holds, then $\UUU \; {\equiv}_{RK} \; \HHH$.
\end{proof}
Theorem \ref{thm:sourkbelow} actually shows that stable ordered-union ultrafilters are RK-mi\-ni\-mal among the ultrafilters on $\Pset(\FIN) \slash \ih$.
The proof of this fact relies on a canonical form for functions modulo restriction to a set belonging to a product of selective ultrafilters.
We will include this argument below because it is not as well-known.
\begin{Lemma} \label{lem:functionprod}
 Suppose $\UUU$ and $\VVV$ are selective ultrafilters on $\omega$.
 Let $f: \omega \times \omega \rightarrow \omega$.
 There exists $A \in \UUU \bigotimes \VVV$ such that one of the following hold:
 \begin{enumerate}
  \item
  $f$ is constant on $A$;
  \item
  $f$ is one-to-one on $A$;
  \item
  there is a one-to-one $g: \omega \rightarrow \omega$ such that $\forall \pr{m}{n} \in A\[f(\pr{m}{n}) = g(m)\]$;
  \item
  there is a one-to-one $g: \omega \rightarrow \omega$ such that $\forall \pr{m}{n} \in A\[f(\pr{m}{n}) = g(n)\]$.
 \end{enumerate}
\end{Lemma}
\begin{proof}
 Define ${f}_{m}(n) = f(\pr{m}{n})$, for all $m, n \in \omega$.
 Since $\VVV$ is selective, there exists ${A}_{m, n} \in \VVV$ so that either $\forall l \in {A}_{m, n}\[{f}_{m}(l) = {f}_{n}(l)\]$ or ${f}_{m}''{A}_{m, n} \cap {f}_{n}''{A}_{m, n} = \emptyset$, for all $m < n < \omega$.
 There also exist ${B}_{n} \in \VVV$ and ${g}_{n}: \omega \rightarrow \omega$ such that ${f}_{n} \restrict {B}_{n} = {g}_{n} \restrict {B}_{n}$ and ${g}_{n}$ is either constant or one-to-one, for all $n \in \omega$.
 Since $\UUU$ is selective, one of the following 4 cases must occur.
 
 Case I: there exist $C \in \UUU$ and $k \in \omega$ such that for each $n \in C$, ${g}_{n}$ is constantly equal to $k$.
 Then (1) occurs as $f$ is constantly $k$ on $A = {\bigcup}_{n \in C}{\left( \{n\} \times {B}_{n} \right)} \in \UUU \bigotimes \VVV$.
 
 Case II: there exist $C \in \UUU$ and a one-to-one function $g: \omega \rightarrow \omega$ such that for each $n \in C$, ${g}_{n}$ is constantly equal to $g(n)$.
 Then (3) holds with $A = {\bigcup}_{n \in C}{\left( \{n\} \times {B}_{n} \right)}$.
 
 Case III: there exists $C \in \UUU$ such that for each $n \in C$, ${g}_{n}$ is one-to-one and for each $m \in C \cap n$, $\forall l \in {A}_{m, n}\[{f}_{m}(l) = {f}_{n}(l)\]$.
 Let $m = \min(C)$ and $g = {g}_{m}$.
 Let ${D}_{m} = {B}_{m}$, for all $n \in C$ with $m < n$, let ${D}_{n} = {B}_{m} \cap {A}_{m, n}$, and let $A = {\bigcup}_{n \in C}{\left( \{n\} \times {D}_{n} \right)} \in \UUU \bigotimes \VVV$.
 Then (4) holds on $A$.
 
 Case IV: there exists $C \in \UUU$ such that for each $n \in C$, ${g}_{n}$ is one-to-one and for each $m \in C \cap n$, ${f}_{m}''{A}_{m, n} \cap {f}_{n}''{A}_{m, n} = \emptyset$.
 Since $\VVV$ is a P-point, find $E \in \VVV$ so that $E \; {\subseteq}^{\ast} \; {A}_{m, n}$, for all $m, n \in C$ with $m < n$.
 For each $n \in C$ find $h(n) \in \omega$ so that for each $m \in C \cap n$, $E \setminus {A}_{m, n} \subseteq h(n)$, and ${g}^{-1}_{n}\left( {f}_{m}''\left( E \setminus {A}_{m, n} \right) \right) \subseteq h(n)$.
 Note that for $m, n \in C$ with $m < n$, ${f}_{m}''E \cap {f}_{n}''\left( {B}_{n} \cap \left( E \setminus h(n) \right) \right) = \emptyset$.
 Now (2) occurs on $A = {\bigcup}_{n \in C}{\left( \{n\} \times \left(  {B}_{n} \cap \left( E \setminus h(n) \right)  \right) \right)} \in \UUU \bigotimes \VVV$.
\end{proof}
\begin{Lemma} \label{lem:sounonprod}
 Suppose $\UUU$ and $\VVV$ are selective ultrafilters on $\omega$ and $\HHH$ is an ultrafilter on $\FIN$ with $\HHH \cap \ih = \emptyset$.
 Then $\HHH \; {\not\equiv}_{RK} \; \UUU \bigotimes \VVV$.
\end{Lemma}
\begin{proof}
 Suppose not.
 Let $f: \omega \times \omega \rightarrow \FIN$ be a one-to-one map witnessing $\HHH \; {\equiv}_{RK} \; \UUU \bigotimes \VVV$.
 Define $$g(\pr{m}{n}) = \min(f(\pr{m}{n}))\ \ \mbox{and}\ \ h(\pr{m}{n}) = \max(f(\pr{m}{n})),$$ for all $\pr{m}{n} \in \omega \times \omega$.
 Lemma \ref{lem:functionprod} applied to the maps $g$ and $h$ together with the hypothesis $\HHH \cap \ih = \emptyset$ imply that there exist one-to-one functions $\varphi, \psi: \omega \rightarrow \omega$ and a set $A \in \UUU \bigotimes \VVV$ such that for each $\pr{m}{n} \in A$, $g(\pr{m}{n}) = \varphi(m)$ and $h(\pr{m}{n}) = \psi(n)$.
 However, there exist $s, t \in f''A$ such that $s \neq t$, $\max(s) = \max(t)$ and $\min(s) = \min(t)$, which gives a contradiction.
\end{proof}
\begin{Cor}[Blass]
 If $\HHH$ is a stable ordered-union ultrafilter on $\FIN$ and $\KKK$ is any ultrafilter on $\FIN$ such that $\KKK \cap \ih = \emptyset$ and $\KKK \; {\leq}_{RK} \; \HHH$, then $\KKK \; {\equiv}_{RK} \; \HHH$.
 In particular, stable-ordered union ultrafilters are RK-minimal among all idempotents in $(\gamma\FIN, \cup)$.
\end{Cor}
\begin{proof}
 Since $\KKK \cap \ih = \emptyset$, the map $s \mapsto \min(s)$ is not finite-to-one or constant on any set in $\KKK$.
 Therefore $\KKK$ is not a P-point and so $\KKK \; {\not\equiv}_{RK} \; {\HHH}_{\min}$ and $\KKK \; {\not\equiv}_{RK} \; {\HHH}_{\max}$.
 By Lemma \ref{lem:sounonprod}, $\KKK \; {\not\equiv}_{RK} \; \UUU \bigotimes \VVV$.
 Therefore by Theorem \ref{thm:sourkbelow}, $\KKK \; {\equiv}_{RK} \; \HHH$.
 
 The second sentence follows from the first, the fact that stable ordered-union ultrafilters are idempotent, and the fact that every idempotent is disjoint from $\ih$.
\end{proof}
\begin{Question} \label{q:sousminimal}
 Suppose $\HHH$ is an idempotent in $(\gamma\FIN, \cup)$ which is RK-minimal among all idempotents in $(\gamma\FIN, \cup)$.
 Is $\HHH$ ordered-union?
\end{Question}
Regarding the Tukey types of stable ordered-union ultrafilters, Dobrinen and Todorcevic~\cite{dt} have raised the following question.
\begin{Question}[Question 56 of \cite{dt}] \label{q:56}
 If $\HHH$ is any stable ordered-union ultrafilter, does it follow that $\HHH \; {>}_{T} \; {\HHH}_{\mathrm{minmax}}$?
\end{Question}
In a recent work, Benhamou and Dobrinen~\cite{benhamou2023tukey} have proved that $\HHH \; {\equiv}_{T} \; \HHH \bigotimes \HHH$, for every stable ordered-union ultrafilter $\HHH$.

We will end this section with a discussion of some cardinal invariants associated with the ultrafilters discussed in this section and in Section \ref{sec:selective}.
Virtually nothing is known about cardinal invariants of the Boolean algebra $\Pset(\FIN) \slash \ih$.
The cardinal invariants associated with such definable quotients have attracted considerable attention in the literature.
Notable examples include~\cite{szzh}, \cite{bs1}, and \cite{MR4518092}, among many others.
\begin{Def} \label{def:uhrh}
 Let $\BBB$ be a non-atomic Boolean algebra.
 Define
 \begin{align*}
  &{\uu}_{\BBB} = \left\{\lc \FFF \rc: \FFF \subseteq \BBB \wedge \FFF \ \text{is a filter base for an ultrafilter on} \ \BBB\right\},\\
  &{\rrr}_{\BBB} = \left\{\lc \FFF \rc: \FFF \subseteq \BBB \setminus \{0\} \wedge \forall a \in \BBB \exists b \in \FFF\[b \leq a \vee b \leq 1-a\]\right\}.
 \end{align*}
 ${\uu}_{\mathtt{H}}$ will denote ${\uu}_{\left( \Pset(\FIN) \slash \ih \right)}$ and ${\rrr}_{\mathtt{H}}$ will denote ${\uu}_{\left( \Pset(\FIN) \slash \ih \right)}$.
\end{Def}
Note that ${\rrr}_{\BBB}$ is the least possible size of a pseudo-base for an ultrafilter on $\BBB$.
Of course, ${\uu}_{\left( \Pset(\omega) \slash \pc{\omega}{< {\aleph}_{0}} \right)}$ and ${\rrr}_{\left( \Pset(\omega) \slash \pc{\omega}{< {\aleph}_{0}} \right)}$ are the cardinals $\uu$ and $\rrr$, which have been well-studied.
However, their relationship with ${\uu}_{\mathtt{H}}$ and ${\rrr}_{\mathtt{H}}$ is unknown.
\begin{Question} \label{q:uhrhur}
 Is it consistent that $\uu < {\uu}_{\mathtt{H}}$?
 Is it consistent that $\rrr < {\rrr}_{\mathtt{H}}$?
\end{Question}
The least possible size of a pseudo-base for a selective ultrafilter may be larger than $\rrr$.
The exact relationship between ${\rrr}_{\mathtt{H}}$ and the minimal number of elements in a pseudo-base for a stable ordered-union ultrafilter is not known.
The following cardinals are well-defined by Ramsey's theorem and the Milliken-Taylor theorem respectively.
\begin{Def} \label{def:hom}
 Define
 \begin{align*}
  &\homr = \min\left\{\lc \FFF \rc: \FFF \subseteq {\[\omega\]}^{{\aleph}_{0}} \wedge \forall c \in {2}^{\left({\[\omega\]}^{2}\right)} \exists A \in \FFF \[c\restrict {\[A\]}^{2} \ \text{is constant}\]\right\}\\
  &\homh = \min\left\{\lc \FFF \rc: \FFF \subseteq {\FIN}^{\[\omega\]} \wedge \forall c \in {2}^{\left({\FIN}^{\[2\]}\right)} \exists X \in \FFF\[c\restrict {\[X\]}^{\[2\]} \ \text{is constant}\]\right\}.
 \end{align*}
\end{Def}
Observe that the cardinality of any pseudo-base for a selective ultrafilter must be at least $\homr$, and that any pseudo-base for a stable ordered-union ultrafilter must contain at least $\homh$ elements.
The cardinal $\homr$ was studied by Blass~\cite{blasssmall}, who proved that $\homr = \max\{{\rrr}_{\sigma}, \dd\}$.
It is clear that $\homr \leq \homh$, but little else is known.
\begin{Question} \label{q:3}
 Is it consistent that $\homr < \homh$?
\end{Question}

\section{P-point ultrafilters}\label{s:p}

Recall that we have defined P-point ultrafilters in Definition \ref{d:ppoint}.
The notion of a P-point in a topological space has its roots in the work of Gillman and Henriksen on rings of continuous functions in \cite{gillman_henriksen}.
Later Rudin constructed a P-point ultrafilter using $\CH$ and used it to show that, consistently, the space of ultrafilters $\omega^*$ is not homogeneous.
Afterwards, in \cite{ketonen}, Ketonen constructed a P-point ultrafilter from the weaker assumption $\mathfrak d=\mathfrak{c}$.
In particular, he showed that under $\mathfrak d=\mathfrak{c}$ every ultrafilter generated by less then $\mathfrak{c}$ many elements is a P-point, by proving the following (in a sense of Definition \ref{def:generic}).

\begin{Theorem}\label{t:ketonen}
	The equality $\mathfrak d=\mathfrak{c}$ holds if and only if P-points generically exist.
\end{Theorem}

Note that some set theoretic assumption is necessary for a P-point ultrafilter to exist.
This had been shown in the work of Shelah and Wimmers by proving that there is a model of $\ZFC$ with no P-points (see \cite{PIF} and \cite{wimmers}).
Recently, Chodounsky and Guzman proved that there are canonical models of set theory, namely Silver's model, where P-points do not exist (see \cite{david_osvaldo}).
The assumption which guarantees the existence of many P-points is Martin's Axiom, $\MA$, as well as $\CH$ implying it. There are $2^{\mathfrak{c}}$ many P-points under $\MA$.

After these initial investigations, there has been a significant amount of work on these, and similar types of ultrafilters. 
Simultaneously, the theory of orderings on ultrafilters has been developed.
An early example of this is the strengthening of Rudin's theorem, a proof that $\omega^*$ is not homogeneous in $\ZFC$. 
This was obtained through the analysis of the Rudin-Frol\'ik order (see \cite{frolik} and \cite{merudin}).
Afterwards, due to its connections with model theory and topology, the Rudin-Keisler ordering became predominantly explored ordering of ultrafilters. 
The first systematic analysis of this ordering on P-points was done by Blass in \cite{blass}.
He extensively used a model theoretic definition of a P-point ultrafilter. For this we will need a few notions.

The language $L$ will consist of symbols for all relations and all functions on $\omega$.
Let $N$ be the standard model for this language, its domain is $\omega$ and each relation or function denotes itself.
Let $M$ be an elementary extension of $N$, and let ${}^*R$ be the relation in $M$ denoted by $R$, and let ${}^*f$ be the function in $M$ denoted by $f$.
Note that if $a\in M$, then the set $\set{{}^*f(a):f\in\omega^{\omega}}$ is the domain of an elementary submodel of $M$.
A submodel like this, i.e. generated by a single element, will be called \emph{principal}.
It is not difficult to prove that a principal submodel generated by $a$ is isomorphic to the ultrapower of the standard model by the ultrafilter ${\cu}_{a} = \{X \subseteq \omega: a \in {}^{\ast}{X}\}$.
If $A,B\subseteq M$, we say that they are \emph{cofinal with each other} iff $\forall a\in A \exists b\in B\ [a\ {}^*\!\!\le b]$ and $\forall b\in B \exists a\in A\ [b\ {}^*\!\!\le a]$.
Now we can state another definition of a P-point ultrafilter.

\begin{Lemma}\label{l:model_p_point}
	An ultrafilter $\cu$ on $\omega$ is a P-point if and only if every nonstandard elementary submodel of $\omega^{\omega}/\cu$ is cofinal with $\omega^{\omega}/\cu$.
\end{Lemma}

There is also a reformulation of the Rudin-Keisler reducibility in model theoretic terms. In particular, for ultrafilters $\cu$ and $\cv$ on $\omega$ (not neccesarily P-points) $\cu\le_{RK}\cv$ if and only if $\omega^{\omega}/\cu$ can be elementary embedded in $\omega^{\omega}/\cv$. 

Using these model theoretic tools Blass was able to prove (in $\ZFC$):

\begin{Theorem}\label{t:blass_lower}
	If $\set{\cu_n:n<\omega}$ is a countable set of P-point ultrafilters such that $\cu_n\le_{RK}\cu_0$ for $n<\omega$, then there is a P-point $\cu$ such that $\cu\le_{RK}\cu_n$ for all $n<\omega$.
\end{Theorem}

In other words, if a countable set of P-points has an $RK$ upper bound which is a P-point, then it also has a lower bound. 
This theorem has two immediate consequences.

\begin{Cor}
	Any $RK$-decreasing $\omega$-sequence of P-points has an $RK$-lower bound which is a P-point. 
\end{Cor}

\begin{Cor}\label{cor:two_p_points}
	If two P-points have an upper bound which is a P-point, then they also have a lower bound.
\end{Cor}

From the last corollary it follows that, since selective ultrafilters are $RK$-minimal P-points (hence, minimal ultrafilters as well), any two $RK$-inequivalent selective ultrafilters do not have an $RK$-upper bound which is a P-point.
Since $\MA$ implies the existence of $2^{\mathfrak c}$ $RK$-inequivalent selective ultrafilters, the last results may be viewed as a witness to the fact that, under $\MA$, the $RK$ ordering of P-points is not upwards directed. 
Note also that the situation described in Corollary \ref{cor:two_p_points} happens, as shown by Blass:

\begin{Theorem}
	Assume $\MA$. There is a P-point ultrafilter with two incomparable $RK$-predecessors (which are P-points by Remark \ref{r:downwardPpoint}).
\end{Theorem}

Blass was also able to prove, under the same assumption of $\MA$, that there is no $RK$-maximal P-point ultrafilter.
This implies that, under $\MA$, there is a an $RK$-increasing $\omega$-sequence of P-points, and in fact, every P-point ultrafilter can be the first element of such a sequence.
Even more is true, as proved by Blass:

\begin{Theorem}\label{t:blass_omega1_real}
	Assume $\MA$. Then:
	\begin{enumerate}
		\item Every $RK$-increasing $\omega$-sequence of P-points has an $RK$-upper bound which is a P-point. In particular, every P-point ultrafilter is the first element of some $RK$-increasing $\omega_1$-sequence of P-points.
		\item There is an order isomorphic embedding of the real line $\br{\mathbb R,\le}$ into the set of P-points under the $RK$-ordering.
	\end{enumerate}
\end{Theorem}

All these results show that the $RK$-ordering of P-points is very rich under suitable set theoretic assumptions.
Let us now review some straightforward obstructions on this ordering.
First, there are only $2^{\mathfrak{c}}$ many ultrafilters on $\omega$, so this ordering can be at most this size. Second, there are only $\mathfrak{c}$ many maps from $\omega$ to $\omega$, so any P-point can have at most $\mathfrak{c}$ many predecessors.
At the moment, there are no other known restrictions, hence the working hypothesis is the following:

\begin{conj}\label{c:workinghypothesis}
	Assume $\MA(\sigma-\mbox{centered})$. Let $(P,<)$ be a partial order of size at most $2^{\mathfrak{c}}$ such that every $p\in P$ has at most $\mathfrak{c}$ many predecessors. Then $P$ embeds into the $RK$ ordering of P-points.
\end{conj}

Note that already Blass asked which orders can be embedded into the set of P-points under natural set-theoretic assumptions.
This particular statement of the conjecture, in the form of a question, was first mentioned by Raghavan and Shelah in \cite{dilip_shelah}.
Note also that the working conjecture is actually a bit stronger.
It conjectures that $P$ from the statement embeds both into the $RK$ and Tukey orderings of P-points.
The analysis of the Tukey order of ultrafilters was initiated in the work of Isbell (see \cite{isbell}), and later revived in the work of Milovich (see \cite{milovich}), although in a bit different setting.
It was then further developed in the work of Dobrinen, Todorcevic and the second author. 

One of the central results of this theory is the following theorem, due to Dobrinen and Todorcevic \cite{dt}, building on the work of Solecki and Todorcevic in \cite{slawek}.

\begin{Theorem}
	Let $\cu$ be a P-point and $\cv$ any ultrafilter. If $\cv\le_T\cu$, then there is a continuous $\phi:\cu\to\cv$ which is monotone and cofinal in $\cv$.
\end{Theorem}

Note that this result shows that the obstruction in Conjecture \ref{c:workinghypothesis} that every element has at most $\mathfrak{c}$ many predecessors is necessary in the Tukey setting as well.

We will now explain some progress on the conjecture.

\subsection{Boolean algebras}

One of the main advances towards Conjecture \ref{c:workinghypothesis} is the proof of the second author and Shelah that, under $\MA(\sigma$-centered), Boolean algebra $(\mathcal P(\omega)/\FIN,\subseteq^*)$ embeds into the set of P-points under both the $RK$ and the Tukey ordering (see \cite{dilip_shelah}). 
In particular, they were able to prove the following.

\begin{Theorem}
	Assume $\MA(\sigma$-centered).
	Then there is a sequence of P-points $\seq{\cu_{[a]}:[a]\in\mathcal P(\omega)/\FIN}$ such that the following two conditions hold:
	\begin{enumerate}
		\item if $a\subseteq^* b$, then $\cu_{[a]}\le_{RK}\cu_{[b]}$;
		\item if $b\not\subseteq^* a$, then $\cu_{[b]}\not\le_T\cu_{[a]}$.
	\end{enumerate}
\end{Theorem}

Since any partial order of size at most the continuum can be embedded into the Boolean algebra $(\mathcal P(\omega)/\FIN,\subseteq^*)$, this result immediately yields a corollary.

\begin{Cor}
	Under $\MA(\sigma$-centered) any partial order of size at most $\mathfrak{c}$ embeds into the set of P-points under both $RK$ and Tukey ordering.
\end{Cor}

\subsection{Lower bounds}

In this subsection we present a strengthening of Theorem \ref{t:blass_lower} in a sense that with a reasonably stronger assumption, we extend the result of that theorem.
These stronger assumptions have two ingredients.
One is that we assume $\MA$, which guarranties the existence of many P-points. 
The other is that the considered set of P-point ultrafilters has a P$_{\mathfrak{c}}$-point as an upper bound.
Recall that an ultrafilter $\cu$ on $\omega$ is \emph{a P$_{\mathfrak{c}}$-point} if for every $\alpha<\mathfrak{c}$ and each collection $\set{a_i:i<\alpha}\subseteq \cu$ there is an $a\in\cu$ such that $a\subseteq^*a_i$ for each $i<\alpha$.
As well as the set of P-points, the set of P$_{\mathfrak{c}}$-points is downward closed with respect to the $RK$-ordering. 
Finally, we can state the main result of this subsection, which is a jont work of the authors with Verner (see \cite{lowerbounds}).

\begin{Theorem}\label{t:krv}
	Assume $\MA_{\alpha}$.
	Suppose that $\set{\cu_i:i<\alpha}$ is a set of P-points such that $\cu_0$ is a P$_{\mathfrak{c}}$-point and that $\cu_i\le_{RK}\cu_0$ for each $i<\alpha$.
	Then there is a P-point $\cu$ such that $\cu\le_{RK}\cu_i$ for each $i<\alpha$.
\end{Theorem}

There is an immediate corollary to this theorem.

\begin{Cor}
	Assume $\MA$.
	Then:
	\begin{enumerate}
		\item If a collection of fewer than $\mathfrak{c}$ many P$_{\mathfrak{c}}$-points has an upper bound which is a P$_{\mathfrak{c}}$-point, then there is a P$_{\mathfrak{c}}$-point which is a lower bound of this collection in the $RK$-ordering. 
		\item The class of P$_{\mathfrak{c}}$-points is downward $<\mathfrak{c}$-closed in the $RK$-ordering.
	\end{enumerate}
\end{Cor}

The thing to emphasize about this result is that it uses model-theoretic analysis of the $RK$-ordering in the same way Blass uses it to prove Theorem \ref{t:blass_lower}.
Namely, the result of Blass is that if $\set{M_i:i<\omega}$ is a collection of pairwise cofinals submodels of $M$ (in the notation from the paragraph just before Lemma \ref{l:model_p_point}) such that at least one of $M_i$'s is principal, then $\bigcap_{i<\omega}M_i$ contains a principal submodel cofinal with each $M_i$ ($i<\omega$).
In \cite{lowerbounds} this result is extended as follows.

\begin{Theorem}\label{t:krv_models}
	Assume $\MA_{\alpha}$.
	Let $\set{M_i:i<\alpha}$ be a collection of pairwise cofinal submodels of $M$.
	Suppose that $M_0$ is a principal submodel and that $\cu_0=\set{X\subseteq \omega:a_0\in\ {}^{*}X}$ is a P$_{\mathfrak{c}}$-point, where $a_0$ generates $M_0$.
	Then $\bigcap_{i<\alpha}M_i$ contains a principal submodel cofinal with each $M_i$.
\end{Theorem}

A natural question here is whether in Theorem \ref{t:krv} and Theorem \ref{t:krv_models} one can remove the assumption of $\cu_0$ being a P$_{\mathfrak{c}}$-point.

\begin{Question}
	Is it consistent with $\ZFC$ that for any $\alpha<\mathfrak{c}$ and any collection of P-points $\set{\cu_i:i<\alpha}$ such that $\cu_i\le_{RK}\cu_0$ for each $i<\alpha$, there is a P-point $\cu$ such that $\cu\le_{RK}\cu_i$.
\end{Question}

At this point we would also like to mention two results about descending chains of P-points.
Both of these prove that there is a P-point with specific properties.
Among other things, it is a rapid ultrafilter.

\begin{Def}\label{d:rapid}
	We say that an ultrafilter $\mathcal U$ on $\omega$ is \emph{rapid} if for every $f\in\omega^{\omega}$ there is $X\in \mathcal U$ such that $X(n)\ge f(n)$ for every $n<\omega$.
\end{Def}

The first of these results is due to Laflamme in \cite{laflamme}, where he is able to generically construct $RK$-descending chains of P-points with very strong properties, but only of arbitrary countable length.
In particular, he proves the following (note that this concise statement is a minor modification of the statement from \cite{natasha_stevo_transactions_2}).

\begin{Theorem}\label{t:laflamme}
	For each $1\le\alpha<\omega_1$, there is an ultrafilter $\cu_{\alpha}$, generic for certain partial order $\mathbb P_{\alpha}$ with the following properties:
	\begin{enumerate}
		\item $\cu_{\alpha}$ is a rapid P-point ultrafilter.
		\item There is a sequence $\seq{\cv_{\gamma}:\gamma<\alpha+1}$ of P-points such that $\cv_0=\cu_{\alpha}$, that $\cv_{\gamma}<_{RK}\cv_{\beta}$ for all $\beta<\gamma<\alpha+1$, and that for any $\cu$ with $\cu\le_{RK}\cu_{\alpha}$ there is $\gamma<\alpha+1$ such that $\cu\equiv_{RK}\cv_{\gamma}$.
	\end{enumerate}
\end{Theorem}

The second one is due to Dobrinen and Todorcevic in \cite{natasha_stevo_transactions_2}.
There, for each countable ordinal, they construct a topological Ramsey space $\mathcal R_{\alpha}$, which then generically gives a specific P-point $\cu_{\alpha}$ and completely describes its Tukey predecessors.
Note that this essentially proves the mentioned Laflamme's result in the Tukey setting.

\begin{Theorem}\label{t:stevonatasha}
	For each $1\le\alpha<\omega_1$, there is a topological Ramsey space $\mathcal R_{\alpha}$, and an ultrafilter $\cu_{\alpha}$, generic for a partial order $(\mathcal R_{\alpha},\le_{\alpha}^*)$ with the following properties:
	\begin{enumerate}
		\item $\cu_{\alpha}$ is a rapid P-point ultrafilter.
		\item There is a sequence $\seq{\cv_{\gamma}:\gamma<\alpha+1}$ of P-points such that $\cv_0=\cu_{\alpha}$, that $\cv_{\gamma}<_{T}\cv_{\beta}$ for all $\beta<\gamma<\alpha+1$, and that for any $\cu$ with $\cu\le_{T}\cu_{\alpha}$ there is $\gamma<\alpha+1$ such that $\cu\equiv_{T}\cv_{\gamma}$.
	\end{enumerate}
\end{Theorem}

Note that it is not necessary to force with $(\mathcal R_{\alpha},\le^*)$ to obtain the ultrafilter $\cu_{\alpha}$.
For example, under $\CH$ or $\MA$, this ultrafilter can be constructed from the Ramsey space $\mathcal R_{\alpha}$.
For more details see \cite[Section 5]{natasha_stevo_transactions_2}.

\subsection{Initial segments}

As we have already pointed out, there has been a lot of work on understanding the order structure of the set of P-points under natural set-theoretic assumptions.
The direction we have described so far concentrated on Conjecture \ref{c:workinghypothesis} and the question which partial orders can be embedded into the set of P-points under the $RK$ and other relevant orderings.
More information can be obtained if one investigates which orders can be embedded as initial segments into the same set.
Recall that $(P,\le_P)$ embeds into $(Q,\le_Q)$ as an initial segment if the image of $P$ under the embedding is a downward closed subset of $Q$ (i.e. the image is an initial segment of $Q$).

Note that the result of Laflamme, Theorem \ref{t:laflamme}, is in this spirit.
It gives an initial segment of P-points which is a reversed successor ordinal, and Theorem \ref{t:stevonatasha}, extends it to the Tukey setting.
However, for the rest of this subsection we focus on increasing initial segments of P-points.

One of the first results in this direction, in the class of all ultrafilters, is again due to Blass in \cite{blass_initial} where he constructed an initial segment of ultrafilters of order type $\omega_1$ under $\CH$.
In other words, using $\CH$ he inductively constructed a sequence $\seq{\cu_{\alpha}:\alpha<\omega_1}$ of ultrafilters such that $\cu_{\alpha}<_{RK}\cu_{\beta}$ for all $\alpha<\beta<\omega_1$ and that for any ultrafilter $\cu$, if $\cu\le_{RK}\cu_{\alpha}$ for some $\alpha<\omega_1$, then there is $\gamma\le\alpha$ such that $\cu\equiv_{RK}\cu_{\gamma}$.

Note that Blass uses the mentioned model theoretic approach in this result as well.
In particular, for a model of complete arithmetic $M'$ he defines its proper elementary extension $M''$ to be \emph{strictly minimal} if every proper submodel of $M''$ is a subset of $M'$.
He then proves, using $\CH$, that it is possible to build an increasing continuous chain of models of arithmetic in which each successor term is strictly minimal extension of its immediate predecessor.
This chain can be then transformed in the required $RK$-initial segment of ultrafilters.

Rosen, in his construction of an $RK$-initial segment of P-points of order type $\omega_1$, also under $\CH$, takes a similar approach.
For a P-point $\cu$, he defines a P-point $\cv$ to be \emph{a strong immediate successor} of $\cu$ if $\cu\le_{RK}\cv$, and for any ultrafilter $\mathcal W$, if $\mathcal W\le_{RK}\cv$, then $\mathcal W\le_{RK}\cu$.

The results of Eck from his PhD Thesis \cite{eck_phd} then take care of the existence of a countable initial segment of P-points.
Namely, Eck provided a method which proves the following theorem.

\begin{Theorem}\label{t:eck}
	Assume $\CH$.
	Let $\cu$ be a P-point ultrafilter.
	Then there is a P-point $\cv$ which is a strong immediate successor of $\cu$.
\end{Theorem}

This construction can be then carried countably many times to obtain an $RK$-initial segment of P-points of order type $\omega$.
Note however, that Rosen was not able to simply use Eck's result in his construction of an uncountable initial segment.
Some care has to be taken in the construction of strong immediate successors in each step, to make it possible for a given countable limit initial segment to be minimaly extended.
In order to precisely state these results, we recall a few more notions from model theory.

If $M'$ is a finitely generated model of complete arithmetic, then we say that $M'$ is \emph{element generated} if for every generator $a$ of $M'$ there is a generator $b$ of $M'$ such that $b\in a(M')$.
Here $a(M')=\set{c\in M':M'\models c\in' a}$, where $\in'$ is defined as follows: for $m,n\in\omega$, we say that $m\in' n$ iff $2^m$ occurs in the binary expansion of $n$.
\emph{A cut} in an ultrapower $\omega^{\omega}/\cu$ is a partition into convex sets $S$ and $L$ such that $a\ {}^{*}\!\!<b$ for every $a\in S$ and $b\in L$.
If $X$ is a set, and $f$ is a function which is finite-to-one on $X$, then we define a function $C_{X,f}$ by $C_{X,f}(n)=\abs{X\cap f^{-1}\br{\set{n}}}$.
If $\cv$ is a P-point and $f:\omega\to\omega$ witnesses that $\cu\le_{RK}\cv$, then the cut in $\omega^{\omega}/\cu$ associated to $f$ and $\cv$ is defined by putting into $L$ all germs $[C_{X,f}]_{\cu}$ for $X\in\cv$ and all larger germs, and defining $S=\br{\omega^{\omega}/\cu}\setminus L$.
Now we are able to state the main technical lemma Rosen uses:

\begin{Theorem}\label{rosen_main}
	Assume $\CH$.
	Let $\set{\cu_i:i<\omega}$ be an $RK$-increasing sequence of P-points, where $f_i$ witnesses that $\cu_i\le_{RK}\cu_{i+1}$, and let $M'$ be the direct limit of the family $\set{\br{\omega^{\omega}/\cu_i,f_i^*}:i<\omega}$.
	For $i\ge 1$ let $g_i=f_0\circ f_1\circ \cdots\circ f_{i-1}$ and let $\br{S^i,L^i}$ be the cut in $\omega^{\omega}/\cu_0$ associated to $g_i$ and $\cu_i$.
	Assume also that $M'$ admits a strictly minimal extension, that $S_i\subseteq S^{i+1}$ and $S^i\neq S^{i+1}$ for $i\ge 1$, and that $\omega^{\omega}/\cu_i$ is element generated for $i\ge 0$.
	Then there is an element generated, strictly minimal extension of $M'$ whose every nonstandard submodel is cofinal with it.
\end{Theorem}

So, in \cite{rosen}, Rosen is using a modification of Eck's result for the successor case which enables him to use Theorem \ref{rosen_main} in order to construct the required initial segment of type $\omega_1$.

\begin{Theorem}
	Assume $\CH$.
	Let $\cu$ be a selective ultrafilter. 
	Then there is an $RK$-initial segment of P-points $\seq{\cu_{\alpha}:\alpha<\omega_1}$ such that $\cu_0=\cu$.
\end{Theorem}

Since Dobrinen and Todorcevic proved in \cite{dt} that, under $\MA$, $\omega_1$ embeds into the Tukey ordering of P-points, the following question suggests itself.

\begin{Question}
	Is there, under suitable set-theoretic assumptions, a Tukey-initial segment of P-points of length $\omega_1$?
\end{Question}

\subsection{Chains} As we have already mentioned, Blass proved that $\omega_1$ embeds into the $RK$ ordering of P-points under suitable hypothesis, while by comments after the statement of Conjecture \ref{c:workinghypothesis} we know that $\mathfrak{c}^+$ is the largest ordinal which can possibly embed into the set of P-points under both $RK$ and Tukey ordering.
In \cite{longchain} the authors proved that this is indeed possible under $\CH$.
Afterwards, the second author and Verner in \cite{verner} improved the construction from \cite{longchain}, and Starosolski contributed by constructing shorter chains but under the weaker assumptions.

First we explain the construction from \cite{longchain}.
There, the authors use the notion of a $\delta$-generic sequence of P-points.
To state the definition we have to introduce a few notions.
First, the fundamental partial order that was used.

\begin{Def}\label{d:orderP}
	Define $\mathbb{P}$ to be the set of all functions $c:\omega\to\FIN$ such that $\abs{c(n)}<\abs{c(n+1)}$ and $c(n)<_{\mathtt{b}} c(n+1)$ for each $n<\omega$.
	If $c,d\in \mathbb{P}$, then $c\le d$ if there is an $l<\omega$ such that $c\le_l d$, where $$c\le_l d \Leftrightarrow \forall m\ge l\ \exists n\ge m\ [c(m)\subseteq d(n)].$$
\end{Def}

For $c\in \mathbb{P}$ we define $\sset(c)=\bigcup_{n<\omega}c(n)$.
Now we introduce the concept of a normal triple.

\begin{Def}\label{d:normaltriple}
	A triple $\seq{\pi,\psi,c}$ is called a \emph{normal triple} if $\pi,\psi\in \omega^{\omega}$, for every $l\le l'<\omega$ we have that $\psi(l)\le \psi(l')$, if $\ran(\psi)$ is infinite, and if $c\in \mathbb{P}$ is such that for $l<\omega$ we have $\pi''c(l)=\set{\psi(l)}$ and for $n\in \omega\setminus \sset(c)$ we have $\pi(n)=0$.
\end{Def}

Before we explain the construction of the $\mathfrak{c}^+$-chain of P-points from \cite{longchain}, let us present a simplified version.
Suppose that we have two P-points $\cu_0$ and $\cu_1$ such that $\cu_0\le_{RK}\cv_1$ and that this is moreover witnessed with a function $\pi_{1,0}$ which is nondecreasing.
Then we could use the poset $\mathbb Q$ of all $q=\seq{c_q,\pi_{q,1},\pi_{q,0}}$ such that $c_q\in\mathbb{P}$ and $\pi_{q,0},\pi_{q,1}\in\omega^{\omega}$, and that for each $i<2$:
\begin{itemize}
	\item $\pi_{q,i}''\sset(c_q)\in\cu_i$,
	\item $\forallbutfin k\in\sset(c_q)\ [\pi_{q,0}(k)=\pi_{1,0}(\pi_{q,1}(k))]$,
	\item there are $\psi_{q,i}\in\BS$ and $b_{q,i}\ge c_q$ such that $\seq{\pi_{q,i},\psi_{q,i},b_{q,i}}$ is a normal triple.
\end{itemize}
The ordering on $\mathbb Q$ would be: $q\le s$ if and only if $c_q\le c_s$ and $\pi_{q,0}=\pi_{s,0}$ and $\pi_{q,1}=\pi_{s,1}$.
Now using CH and a set of conditions meeting enough dense sets we would obtain a sequence $\seq{c_{q_{\alpha}}:\alpha<\mathfrak{c}}$ generating a P-point $\cu$ such that $\cu_0\le_{RK}\cu$ and $\cu_1\le_{RK}\cu$.
Moreover, both of these inequalities would be witnessed by nondecreasing maps.

The poset $\mathbb Q$ is a simplified version of the countably closed poset $\mathbb Q^{\delta}$ from Definition \ref{d:theposet} which is used to extend a given $\delta$-generic sequence from Definition \ref{d:genericsequence}.

Note that if only one P-point $\cu$ is given, then just with the poset consisting of pairs $q=\seq{c_q,\pi_q}$ where $c_q\in\mathbb{P}$ and $\pi_q\in\omega^{\omega}$ are such that $\pi_q''\sset(c_q)\in \cu$ and that there are $\psi_q\in\omega^{\omega}$ and $b_q\ge c_q$ such that $\seq{\pi_q,\psi_q,b_q}$ is a normal triple (and the order analoguous to the order of $\mathbb Q$), using $\CH$ and a sequence meeting appropriate dense sets, one can construct a P-point $\cv$ such that $\cu\le_{RK}\cv$.
Moreover, if one chooses dense sets a bit more carefully, then the P-point $\cv$ can be forced to be a strong immediate successor of $\cu$, i.e. $\cv$ then satisfies the condition of Theorem \ref{t:eck} with respect to $\cu$.

Now we are ready to state the definition of a $\delta$-generic sequence of ultrafilters.
We will comment in more detail on some key properties, the role of the remaining notions is more or less clear.
For example, clause (\ref{i5}) means that the maps $\pi_{\beta,\alpha}$ are Rudin-Keisler maps, while the commutativity of these RK maps as stated in clause (\ref{i0}) is unavoidable in a chain.
Clause (\ref{i6}) provides another feature of the sequence produced in this way.
It is a chain in the stronger ordering: every $\omega_2$-generic sequence of P-points is increasing in $\le_{RB}^+$ ordering, i.e. one can choose a function witnessing the $RK$ comparability to be nondecreasing.

The central part of the definition are clauses (\ref{i1}) and (\ref{i8}).
Their role is to anticipate conditions which will be required by future ultrafilters.
Let us try to explain (\ref{i1}) in a simplified setting.
So assume that $\langle {\cu}_{n}: n < \omega \rangle$ has already been built, and that we are now constructing the ultrafilter ${\cu}_{\omega}$.
Suppose for example that we want $\sset(d)$ to be in ${\cu}_{\omega}$, for some $d \in \mathbb{P}$, and that the maps $\langle {\pi}_{\omega, i}: i \leq n \rangle$ are constructed, for some $n \in \omega$.
In particular, we assume ${\pi}_{\omega, n}''\sset(d) \in {\cu}_{n}$.
Now one wants to define what is ${\pi}_{\omega, n + 1}$, and is allowed to strengthen $d$ to some ${d}^{\ast} \leq d$.
However, it must hold that ${\pi}_{\omega, n + 1}''\sset({d}^{\ast}) \in {\cu}_{n + 1}$ and that ${\pi}_{\omega, n }$ commutes through ${\pi}_{\omega, n + 1}$.
Clause (\ref{i1}) is defined in such a way that ${\cu}_{n + 1}$ had this in mind, so there are cofinally many $b \in {\cu}_{n + 1}$ making this possible.

Next, we discuss (\ref{i8}).
So again, assume that $\langle {\cu}_{\alpha}: \alpha < {\omega}_{1} \rangle$ has been built and that we are constructing ${\cu}_{{\omega}_{1}}$.
Supose that for some decreasing sequence $\langle {d}_{n}: n < \omega \rangle$ of conditions in $\mathbb{P}$, we have already decided $\langle \sset({d}_{n}): n < \omega \rangle \subseteq {\cu}_{{\omega}_{1}}$.
Suppose also that $\langle {\pi}_{{\omega}_{1}, n}: n < \omega \rangle$ has been defined.
In particular we assume that ${\pi}_{{\omega}_{1}, n}''\sset({d}_{m}) \in {\cu}_{n}$ for all $n,m<\omega$, and that each ${\pi}_{{\omega}_{1}, n}$ is as it should be on some ${d}_{m}$.
Now we need to determine a ${d}^{\ast} \in \mathbb{P}$ which is stronger than all ${d}_{n}$, and also define the map ${\pi}_{{\omega}_{1}, \omega}$ in such a way that ${\pi}_{{\omega}_{1}, \omega}''\sset({d}^{\ast}) \in {\cu}_{\omega}$, that ${\pi}_{{\omega}_{1}, \omega}$ is of the right form on ${d}^{\ast}$, and that all of the ${\pi}_{{\omega}_{1}, n}$ commute through ${\pi}_{{\omega}_{1}, \omega}$.
Clause (\ref{i8}) is defined is such a way that ${\cu}_{\omega}$ had this in mind, so there are cofinally many $b \in {\cu}_{\omega}$ making this possible.

\begin{Def}\label{d:genericsequence}
	Let $\delta\le \omega_2$ . We call $\seq{\seq{c^{\alpha}_i:i<\mathfrak{c}\wedge\alpha<\delta},\seq{\pi_{\beta,\alpha}:\alpha\le\beta<\delta}}$ \emph{$\delta$-generic} if and only if:
	\begin{enumerate}
		\item\label{i2} for every $\alpha<\delta$, $\seq{c_i^{\alpha}:i<\mathfrak{c}}$ is a decreasing sequence in $\mathbb{P}$; for every $\alpha \leq \beta < \delta$, ${\pi}_{\beta, \alpha} \in \BS$;
		\item\label{i4} for every $\alpha<\delta$, $\mathcal U_{\alpha}=\set{a\in \mathcal P(\omega):\exists i<\mathfrak{c}\ [\sset(c^{\alpha}_i)\subseteq^*a]}$ is an ultrafilter on $\omega$ and it is a rapid P-point (we say that $\mathcal U_{\alpha}$ is \emph{generated by} $\seq{c_i^{\alpha}:i<\mathfrak{c}}$);
		\item\label{i1} for every $\alpha<\beta<\delta$, every normal triple $\seq{\pi_1,\psi_1,b_1}$ and every $d\le b_1$ if $\pi''_1\sset(d)\in\cu_{\alpha}$, then for every $a\in\cu_{\beta}$ there is $b\in\cu_{\beta}$ such that $b\subseteq^* a$ and that there are $\pi,\psi\in\omega^{\omega}$ and $d^*\le_0 d$ so that $\seq{\pi,\psi,d^*}$ is a normal triple, $\pi''\sset(d^*)=b$ and $\forall k\in \sset(d^*)\ [\pi_1(k)=\pi_{\beta,\alpha}(\pi(k))]$.
		\item\label{i9} if $\alpha<\beta<\delta$, then $\cu_{\beta}\nleq_T\cu_{\alpha}$.
		\item\label{i3} for every $\alpha<\delta$, $\pi_{\alpha,\alpha}=\operatorname{id}$ and:
		\begin{enumerate}
			\item\label{i5} $\forall \alpha\le \beta<\delta)\ \forall i<\mathfrak{c}\ [\pi_{\beta,\alpha}''\sset(c^{\beta}_{i})\in \mathcal U_{\alpha}]$;
			\item\label{i0} $\forall \alpha\le \beta\le \gamma<\delta\ \exists i<\mathfrak{c}\ \forall^{\infty} k\in \sset(c_i^{\gamma})\ [\pi_{\gamma,\alpha}(k)=\pi_{\beta,\alpha}(\pi_{\gamma,\beta}(k))]$;
			\item\label{i6} for $\alpha<\beta<\delta$ there are $i<\mathfrak{c}$, $b_{\beta,\alpha}\in\mathbb{P}$ and $\psi_{\beta,\alpha}\in \omega^{\omega}$ such that $\seq{\pi_{\beta,\alpha},\psi_{\beta,\alpha},b_{\beta,\alpha}}$ is a normal triple and $c^{\beta}_i\le b_{\beta,\alpha}$;
		\end{enumerate}
		\item\label{i8} if $\mu<\delta$ is a limit ordinal such that $\cf(\mu)=\omega$, $X \subseteq \mu$ is a countable set such that $\sup(X) = \mu$, $\seq{d_j:j<\omega}$ is a decreasing sequence of conditions in $\mathbb{P}$, and $\seq{\pi_{\alpha}:\alpha\in X}$ is a sequence of maps in $\omega^{\omega}$ such that:
		\begin{enumerate}
			\item\label{i21} $\forall \alpha\in X\ \forall j<\omega\ [\pi_{\alpha}''\sset(d_j)\in \mathcal U_{\alpha}]$;
			\item\label{i22} $\forall \alpha,\beta\in X\ [\alpha\le \beta\Rightarrow \exists j<\omega\ \forall^{\infty} k\!\in\!\sset(d_j)\ [\pi_{\alpha}(k)\!=\!\pi_{\beta,\alpha}(\pi_{\beta}(k))]]$;
			\item\label{i23} for all $\alpha\in X$ there are $j<\omega$, $b_{\alpha}\in \mathbb{P}$ and $\psi_{\alpha}\in\omega^{\omega}$ such that $\seq{\pi_{\alpha},\psi_{\alpha},b_{\alpha}}$ is a normal triple and $d_j\le b_{\alpha}$;
		\end{enumerate}
		then the set of all $i^*<\mathfrak{c}$ such that there are $d^*\in \mathbb{P}$ and $\pi, \psi \in \omega^{\omega}$ satisfying:
		\begin{enumerate}[resume]
			\item\label{i31} $\sset(c^{\mu}_{i^*})=\pi''\sset(d^*)$ and $\forall j<\omega\ [d^*\le d_j]$;
			\item\label{i32} $\forall \alpha\in X\ \forall^{\infty} k\in \sset(d^*)\ [\pi_{\alpha}(k)=\pi_{\mu,\alpha}(\pi(k))]$;
			\item\label{i34} $\seq{\pi,\psi,d^*}$ is a normal triple;
		\end{enumerate}
		is cofinal in $\mathfrak{c}$;
	\end{enumerate}
\end{Def}

Clearly an $\omega_2$-generic sequence witnesses that $\mathfrak{c}^+$ embeds into both $RK$ and Tukey orderings of P-points, under $\CH$.
It was constructed by extending (for $\delta<\omega_2$) a given $\delta$-generic sequence $$S=\seq{\seq{c^{\alpha}_i:i<\mathfrak{c}\wedge\alpha<\delta},\seq{\pi_{\beta,\alpha}:\alpha\le\beta<\delta}},$$ using a sufficiently generic filter over the following countably closed partial order (depending only on $S$). 

\begin{Def}\label{d:theposet}
	Let $\mathbb Q^{\delta}$ be the set of all $q=\seq{c_q,\gamma_q,X_q,\seq{\pi_{q,\alpha}:\alpha\in X_q}}$ such that:
	\begin{enumerate}
		\item\label{i11} $c_q\in\mathbb{P}$;
		\item\label{i12} $\gamma_q\le \delta$;
		\item\label{i13} $X_q\in[\delta]^{\le\omega}$ is such that $\gamma_q=\sup(X_q)$ and $\gamma_q\in X_q$ iff $\gamma_q<\delta$;
		\item\label{i14} $\pi_{q,\alpha}$ ($\alpha\in X_q$) are mappings in $\omega^{\omega}$ such that:
		\begin{enumerate}
			\item\label{i16} $\pi''_{q,\alpha}\sset(c_q)\in \mathcal U_{\alpha}$;
			\item\label{i17} $\forall \alpha,\beta\in X_q\ \left[\alpha\le \beta\Rightarrow \forall^{\infty}k\in \sset(c_q)\ [\pi_{q,\alpha}(k)=\pi_{\beta,\alpha}(\pi_{q,\beta}(k))]\right];$
			\item\label{i15} there are $\psi_{q,\alpha}\in\omega^{\omega}$ and $b_{q,\alpha}\ge c_q$ such that $\seq{\pi_{q,\alpha},\psi_{q,\alpha},b_{q,\alpha}}$ is a normal triple;
		\end{enumerate}
	\end{enumerate}
	The ordering on $\mathbb Q^{\delta}$ is given by: $q_1\le q_0$ if and only if
$$\textstyle
c_{q_1}\le c_{q_0}\ \mbox{and}\ X_{q_1}\supseteq X_{q_0}\ \mbox{and for every}\ \alpha\in X_{q_0},\ \pi_{q_1,\alpha}=\pi_{q_0,\alpha}.
$$
\end{Def}

In particular, it was proved in \cite{longchain} that for any $\delta<\omega_2$ and every $\delta$-generic sequence $S_{\delta}$, there is an $\omega_2$-generic sequence $S$ such that $S\upharpoonright\delta=S_{\delta}$.
Thus we have the following theorem.

\begin{Theorem}
	Assume $\CH$.
	The ordinal $\mathfrak{c}^+$ embeds into both the $RK$ and the Tukey ordering of P-points.
\end{Theorem}

Now that we know that the longest possible ordinal embeds into these two orderings of P-points, the following two questions seem very natural.

\begin{Question}
	Is it true, under suitable set-theoretic assumptions, that $\mathfrak{c}^+$ embeds as an initial segment into the $RK$ ordering of P-points.
\end{Question}

\begin{Question}
	Is it true, under suitable set-theoretic assumptions, that $\mathfrak{c}^+$ embeds as an initial segment into the Tukey ordering of P-points.
\end{Question}

It was shown by the second author and Verner in \cite{raghavan_verner_v1} that conditions (\ref{i1}) and (\ref{i8}) in Definition \ref{d:genericsequence} are redundant.
In particular, if \emph{a weakly $\delta$-generic} sequence is defined as a sequence satisfying clauses (\ref{i2}), (\ref{i4}), (\ref{i9}), and (\ref{i3}) of Definition \ref{d:genericsequence}, then $S$ being $\delta$-generic is equivalent to $S$ being weakly $\delta$-generic.

We would like to point out here that not every $RK$ increasing $\omega_1$-sequence of P-points can be extended, as proved by the second author and Verner in \cite{verner}.
Namely, using $\diamondsuit$, they recursively construct an $RK$ increasing $\omega_1$-sequence of P-points $\seq{\cu_{\gamma}:\gamma<\omega_1}$ which cannot have a P-point on top.
To succeed, they start with arbitrary P-point, and in every successor step, for already constructed sequence $\seq{\cu_{\gamma}:\gamma\le\alpha}$ they choose arbitrary P-point $RK$ above $\cu_{\alpha}$.
In the case when $\seq{\cu_{\gamma}:\gamma<\alpha}$ for $\alpha$ limit is given, using diamond sequence, they define a countably closed forcing which contains a sufficiently generic P-filter, and any extension of this P-filter will be the required P-point ultrafilter.

In the same paper \cite{verner}, they proved the following.

\begin{Theorem}\label{t:raghavanverner}
	Assume $\CH$.
	Let $\delta<\omega_2$ and let $\seq{\cu_{\gamma}:\gamma<\delta}$ be an $RK$ increasing sequence of rapid P-points.
	Then there is a rapid P-point such that $\cu_{\gamma}\le_{RK}\cu$ for every $\gamma<\delta$.
\end{Theorem}

This is clearly stronger result than the existence of an $\omega_2$-generic sequence from \cite{longchain}, in a sense that no care is needed while constructing the approximations to the resulting sequence. 
Note however, that the sequence obtained using Theorem \ref{t:raghavanverner} need not be increasing in the stronger ordering $\le_{RB}^+$.
Hence, in this sense, the construction of that sequence is weaker than the construction of an $\omega_2$-generic sequence from Definition \ref{d:genericsequence}.

Theorem \ref{t:raghavanverner} can be adapted to work under $\MA$ as well.
In this case it would prove that any $RK$ increasing $\mathfrak{c}$-sequence of rapid P-points can be extended with a rapid P-point.
However, it is not clear what is the optimal assumption for this sort of a result.
For example, in \cite{verner}, the authors ask whether $\mathfrak b=\mathfrak{c}$ is sufficient to get the conclusion of Theorem \ref{t:raghavanverner}. 

The work of Starosolski in \cite{starosolski} improves many known results about the $RK$ ordering of P-points to this hypothesis.
For example, he was able to prove the following.

\begin{Theorem}
	Assume $\mathfrak b=\mathfrak c$.
	Then:
	\begin{enumerate}
		\item If $\seq{\cu_n:n<\omega}$ is an $RK$-increasing sequence of P-point ultrafilters, then there is a P-point $\cu$ such that $\cu_n<_{RK}\cu$ for each $n<\omega$.
		\item For each P-point ultrafilter $\cu$, there is an embedding of both the real line and the long line in the $RK$-ordering of P-points above $\cu$.
		\item For every P-point $\cu$ and every $\gamma<\mathfrak{c}^+$ there is an $RK$-increasing sequence of P-points $\seq{\cu_{\alpha}:\alpha<\gamma}$ such that $\cu_0=\cu$.
	\end{enumerate}
\end{Theorem}

\section{Weakenings of being a P-point}\label{sec:weakppoints}

In this section we mention some results about generalizations of the notion of a P-point ultrafilter.
We consider two generalizations, the first with respect to the Tukey ordering and the other related to being generic over a natural partial order. 

Before we proceed with these, we would like to mention another well known generalization of being a P-point.
\emph{A weak P-point} was defined by Kunen in \cite{kunenweak} as an ultrafilter $\cu$ which is not limit of any countable subset of $\omega^*\setminus\set{\cu}$.
Equivalently, $\cu$ is a weak P-point if for any countably many non-principal ultrafilters $\cv_n$ ($n<\omega$), each different from $\cu$, there is $a\in\cu$ such that $a\notin \cv_n$ for $n<\omega$.
In the mentioned paper, it was proved that there are, in $\ZFC$, $2^{\mathfrak{c}}$ many weak P-points in $\omega^*$.

\subsection{Isbell's problem}

As we have already mentioned, Kunen proved, in $\ZFC$, that there are two ultrafilters which are $RK$ incomparable.
Isbell asked the same question in the Tukey context.
Note that every ultrafilter is a directed set of cardinality $\mathfrak{c}$ and that $[\mathfrak{c}]^{<\omega}$ is the maximal Tukey type among directed sets of cardinality at most $\mathfrak{c}$.
Isbell proved that, in $\ZFC$, there is an ultrafilter $\cu$ such that $(\cu,\supseteq)\equiv_T [\mathfrak{c}]^{<\omega}$.
Hence, his question about the existence of two Tukey incomparable ultrafilters reduces to the following.

\begin{Question}
	Is there an ultrafilter $\cu$ on $\omega$ such that $\cu<_T [\mathfrak{c}]^{<\omega}$.
\end{Question}

This is the same as asking whether it is consistent with $\ZFC$ that for every ultrafilter on $\omega$ we have $[\mathfrak{c}]^{<\omega}\le_T\cu$, or equivalently $\cu\equiv_T [\mathfrak{c}]^{<\omega}$.
Note that for every P-point $\cu$ we do have $\cu<_T [\mathfrak{c}]^{<\omega}$.
The reason for this is that in a P-point ultrafilter there are many infinite bounded subsets.
Hence, one can say that, in a sense, any ultrafilter $\cv$ satisfying $\cv<_T [\mathfrak{c}]^{<\omega}$ is close to being a P-point.
We proceed to describe one such class.

Since any ultrafilter is a subset of $\mathcal P(\omega)$, it is a separable metric space with the metric inherited from the Cantor space.
Hence, we can talk about the convergence of a sequence of elements in an arbitrary ultrafilter $\cu$: a sequence $\seq{a_n:n<\omega}$ in $\cu$ converges to some $a\in\cu$ iff for every $m<\omega$ there is some $k<\omega$ such that $a\cap m=a_n\cap k$ for each $n\ge k$.

\begin{Def}\label{def:basicgen}
	An ultrafilter $\cu$ on $\omega$ is \emph{basically generated} if it has a filter basis $\mathcal B\subseteq \cu$ such that every sequence $\seq{a_n:n<\omega}$ in $\mathcal B$ converging to an element of $\mathcal B$ has a subsequence $\seq{a_{n_k}:k<\omega}$ such that $\bigcap_{k<\omega}a_{n_k}\in\cu$.
\end{Def}

Dobrinen and Todorcevic were able to prove in \cite{dt} that each basically generated ultrafilter is weakening of a P-point ultrafilter.

\begin{Theorem}
	If $\cu$ is a basically generated ultrafilter on $\omega$, then $\cu<_T [\mathfrak{c}]^{<\omega}$.
\end{Theorem}

Although the last result shows that basically generated is a generalization of being a P-point in a certain precise sense, it is also possible to prove this in a much more direct sense, as the next theorem shows.
It was also proved in \cite{dt}.

\begin{Theorem}\label{t:ppointbasicallyg}
	Every P-point ultrafilter is basically generated.
\end{Theorem}

Theorem \ref{t:ppointbasicallyg} suggested a question whether there is an ultrafilter $\cu$ which is not basically generated but such that still $\cu<_T [\mathfrak{c}]^{<\omega}$.
This was answered in the negative by Blass, Dobrinen, and Raghavan in \cite{blass_dilip_natasha}.
The result follows from Theorem \ref*{t:smallTukeytype}, Theorem \ref*{t:canoniconcg} and Theorem \ref*{t:cgnotbasicallygenerated} in our presentation.

\subsection{Ultrafilters generic over a Fubini product of ideals}

In this subsection we consider ultrafilters over $\omega^k$ for an integer $k\ge 2$, however, it is clear that using a bijection between $\omega$ and $\omega^k$ this is equivalent to considering ultrafilters on $\omega$.
This approach makes the presentation simpler.
For $l<k$, let $\pi_l:\omega^k\to\omega$ denote the projection to the first $l$ many coordinates, i.e. $\pi_l(x)=\seq{x(0),\dots,x(l-1)}$ for $x\in\omega^k$.
Recall that the ideals $\FIN^{\otimes k+1}$ for $k<\omega$ are defined recursively: $\FIN^{\otimes 1}$ is $\FIN$, while for $k>1$:
$$\FIN^{\otimes k+1}=\set{X\subseteq \omega^{k+1}:\abs{\set{n<\omega:\set{j\in\omega^k:\langle n\rangle^{\frown}j\in X}\notin\FIN^{\otimes k}}}<\omega}.$$
Let $\mathcal{G}_k$ be a filter generic over $V$ for $\mathcal P(\omega^k)/\FIN^{\otimes k}$.
Then in $V[\mathcal{G}_k]$, the filter $\mathcal{G}_k$ is an ultrafilter on $\omega^k$.
As shown in \cite{blass_dilip_natasha}, the ultrafilter $\mathcal{G}_2$ is not a P-point because there is no set $A$ in $\mathcal{G}_2$ such that $\pi_0$ is either constant or finite-to-one on $A$.
However, its projection $\pi_0''\mathcal{G}_2$ is a selective ultrafilter, and moreover it is generic for $\mathcal P(\omega)/\FIN$.
This ultrafilter $\mathcal{G}_2$ has many interesting properties.
For example, although it is not a P-point it is a generalization in a sense of Kunen.

\begin{Theorem}
	The generic ultrafilter $\mathcal{G}_2$ is a weak P-point.
\end{Theorem}

The ultrafilter $\mathcal{G}_2$ is also a counterexample to Isbell's problem.
In particular we have the following theorem.

\begin{Theorem}\label{t:smallTukeytype}
	There are at most continuum many ultrafilters $\cu$ such that $\cu\le_T\mathcal{G}_2$.
	Hence, $\mathcal{G}_2<_T [\mathfrak{c}]^{<\omega}$.
\end{Theorem}

The ultrafilter $\mathcal{G}_2$ is actually a counterexample to Isbell's problem in a strong sense, that even $[\omega_1]^{<\omega}$ is not Tukey reducible to it, as shown by the next theorem.

\begin{Theorem}
	$([\omega_1]^{<\omega},\subseteq)\not\le_T (\mathcal{G}_2,\supseteq)$.
\end{Theorem}

In the next part of this subsection we explain how $\mathcal{G}_2$ is close to being a basically generated ultrafilter, yet it fails to have that property.
Note that Theorem \ref{thm:basic} and \ref{t:canoniconcg} are the same except that basically generated is replaced with generic for $\mathcal P(\omega^2)/\FIN^{\otimes 2}$.

\begin{Theorem}\label{t:canoniconcg}
	Let $\cu$ be an ultrafilter in $V[\mathcal{G}_2]$ such that $\cu\le_T \mathcal{G}_2$.
	Then there is $P\subseteq [\omega^2]^{<\omega}\setminus\set{\emptyset}$ such that 
	\begin{enumerate}
		\item $\forall s,t\in P[t\subseteq s\Rightarrow t=s]$,
		\item $\mathcal{G}(P)\equiv_T \mathcal{G}_2$,
		\item $\cu\le_{RK}\mathcal{G}(P)$.
	\end{enumerate} 
\end{Theorem}

Note that this last result also implies Theorem \ref{t:smallTukeytype}.
However, as the next theorem shows, $\mathcal{G}_2$ is not basically generated.

\begin{Theorem}\label{t:cgnotbasicallygenerated}
  In $V[\mathcal{G}_2]$, the ultrafilter $\mathcal{G}_2$ is not basically generated.
\end{Theorem}

After the work we have just explained, Dobrinen in \cite{natasha_tukey_non_p} investigated ultrafilters $\mathcal{G}_k$ for $k>2$.
Among many relevant results there, we would like to emphasize the characterization of all the ultrafilters Tukey reducible to $\mathcal{G}_k$ for each $k\ge 2$.
The proof is contained in \cite[Sections 6]{natasha_tukey_non_p}.
Note also that canonical form of a monotone map from $\mathcal{G}_k$ to $\mathcal P(\omega)$ has been obtained in Section 5 of the mentioned paper.
However, the presentation of that result would require introduction of many notions, so we just formulate the announced characterization.

\begin{Theorem}\label{t:natashatukeycharacterization}
  Let $k\ge 2$.
  Then $\pi_l''\mathcal{G}_k<_T \pi_{l+1}''\mathcal{G}_k$ for each $l<k-1$.
  Moreover, if $\cv$ is any non-principal ultrafilter in $V[\mathcal{G}_k]$ such that $\cv\le_T\mathcal{G}_k$, then $\cv\equiv_T \pi_l''\mathcal{G}_k$ for some $l<k$.
\end{Theorem}

This answered a problem left open in \cite{blass_dilip_natasha}: where exactly is $\mathcal{G}_2$ is in the Tukey order of ultrafilters?
By Theorem \ref{t:natashatukeycharacterization}, it is minimal over the projection $\pi_0''\mathcal{G}_2$, which is selective.
Hence $\mathcal{G}_2$ has exactly one Tukey predecessor, while $\mathcal{G}_k$ has exactly $k-1$ Tukey predecessors.
Note that analoguous result to Theorem \ref{t:natashatukeycharacterization}, for the Rudin-Keisler reducibility holds as well \cite[Theorem 6.5]{natasha_tukey_non_p}.

\begin{Theorem}
  Let $k\ge 2$.
  If $\cv$ is a non-principal ultrafilter in $V[\mathcal{G}_k]$ such that $\cv\le_{RK}\mathcal{G}_k$, then $\cv\equiv_{RK}\pi_l''\mathcal{G}_k$ for some $l<k$.
\end{Theorem}

\providecommand{\bysame}{\leavevmode\hbox to3em{\hrulefill}\thinspace}
\providecommand{\MR}{\relax\ifhmode\unskip\space\fi MR }
\providecommand{\MRhref}[2]{%
  \href{http://www.ams.org/mathscinet-getitem?mr=#1}{#2}
}
\providecommand{\href}[2]{#2}

\end{document}